\documentclass[10pt]{article}

\usepackage{amssymb}
\usepackage{amsmath}
\usepackage{theorem}
\usepackage{epsfig}
\usepackage{verbatim}
\usepackage{graphicx}
\usepackage{subfigure}
\usepackage{cancel}

\usepackage{color}

\newtheorem{theorem}{Theorem}[section]
\newtheorem{lemma}[theorem]{Lemma}

\newtheorem{prop}[theorem]{Proposition}
\newtheorem{corollary}[theorem]{Corollary}

\newtheorem{defi}[theorem]{Definition}
\newtheorem{rem}[theorem]{Remark}

\newcommand{\R}{\Bbb{R}}

\newcommand{\T}{\mathbb{T}}

\newcommand{\al}{\alpha}

\newcommand{\pa}{\partial}

\newcommand{\Dr}{\Delta R}
\newcommand{\Dh}{\Delta h}
\newcommand{\oR}{\overline{R}}
\newcommand{\oh}{\overline{h}}

\newenvironment{proof}{\begin{trivlist} \item[] {\em Proof:}}{\hfill $\Box$
                       \end{trivlist}}

\makeatletter
\renewcommand*\l@section{\@dottedtocline{1}{0em}{1.5em}}
\renewcommand*\l@subsection{\@dottedtocline{2}{1.5em}{2.3em}}
\renewcommand*\l@subsubsection{\@dottedtocline{3}{3.8em}{3.7em}}
\makeatother

\numberwithin{equation}{section}

\allowdisplaybreaks[4]

\begin{document}

\title{Existence and regularity of rotating global solutions for the generalized surface quasi-geostrophic equations}

\author{Angel Castro, Diego C\'ordoba and Javier G\'omez-Serrano}

\maketitle

\begin{abstract}

Motivated by the recent work of Hassainia and Hmidi [Z. Hassainia, T. Hmidi - On the {V}-states for the generalized quasi-geostrophic equations,\emph{arXiv preprint arXiv:1405.0858}], we close the question of the existence of convex global rotating solutions for the generalized surface quasi-geostrophic equation for $\alpha \in [1,2)$. We also show $C^{\infty}$ regularity of their boundary for all $\al \in (0,2)$. \\


\end{abstract}


\section{Introduction}

In this paper, we consider the generalized surface-quasigeostrophic equation (gSQG):
\begin{align*}
\left\{ \begin{array}{ll}
\partial_{t}\theta+u\cdot\nabla\theta=0,\quad(t,x)\in\mathbb{R}_+\times\mathbb{R}^2, &\\
u=-\nabla^\perp(-\Delta)^{-1+\frac{\alpha}{2}}\theta,\\
\theta_{|t=0}=\theta_0,
\end{array} \right.
\end{align*}

where $\alpha \in (0,2)$. The case $\alpha = 1$ corresponds to the surface quasi-geostrophic (SQG) equation and the limiting case $\alpha = 0$ refers to the 2D incompressible Euler equation. $\alpha = 2$ produces stationary solutions.

Our goal in this article is to show the existence of global rotating solutions (also known as V-states) of the gSQG equation. These solutions are also known to exist for the vortex patch problem ($\al = 0$), see the paper \cite{Hmidi-Mateu-Verdera:rotating-vortex-patch} by Hmidi, Mateu and Verdera, and recently their existence has been shown for the case $0 < \al < 1$ \cite{Hassainia-Hmidi:v-states-generalized-sqg} by Hassainia and Hmidi.

Motivated by the articles of Constantin et al. \cite{Constantin-Majda-Tabak:formation-fronts-qg} and Held et al. \cite{Held-Pierrehumbert-Garner-Swanson:sqg-dynamics}, a lot of effort has been devoted to understanding these equations for the SQG $(\al = 1)$ case. More generally, the problem of whether the gSQG system presents global solutions or not is not completely understood.

  The existence of global weak solutions in $L^2$ for the case $\al = 1$ (SQG) was proven by Resnick in \cite{Resnick:phd-thesis-sqg-chicago} using an extra cancellation due to the oddness of the Riesz transform and it was extended to the gSQG case in \cite{Chae-Constantin-Cordoba-Gancedo-Wu:gsqg-singular-velocities} by Chae et al., even though the question of non-uniqueness is still open (see \cite{Isett-Vicol:holder-continuous-active-scalar} and references therein). A one-dimensional model of the gSQG equations was studied by Dong and Li in \cite{Dong-Li:one-dimensional-alpha-patch}.

 A particular kind of weak solutions for an active scalar are the so called $\alpha-$\textit{patches}, i.e., solutions for which $\theta$ is
 a step function:

  \begin{align}
  \theta(x,t) =
  \left\{
 \begin{array}{ll}
   \theta_1, \text{ if } \ \ x \in \Omega(t) \\ 
   \theta_2, \text{ if }  \ \ x \in \Omega(t) ^c, \\
   \end{array}
  \right.
\end{align}

 where $ \Omega(0)$ is given by the initial distribution of $\theta$, $\theta_1$ and $\theta_2$ are constants, and $\Omega(t)$ is the evolution of $\Omega(0)$ under the velocity field $u$ given by $u(x,t)  = \nabla^\perp \Lambda^{-(2-\alpha)} \theta (x,t) $ for $\Lambda = (-\Delta)^{1/2}$.

  The evolution of such distribution is completely determined by the evolution of the boundary, allowing the problem to be treated as a non-local one dimensional equation for the contour of $\Omega(t)$. In this setting, local existence of smooth solutions was first obtained for $C^{\infty}$ curves by Rodrigo in \cite{Rodrigo:evolution-sharp-fronts-qg} for $0 < \al \leq 1$. The question of local existence of simply connected  $\alpha-$\textit{patches} with Sobolev regularity of its boundary was addressed by Gancedo in \cite{Gancedo:existence-alpha-patch-sobolev} for $0 < \al \leq 1$ and for $1 < \al < 2$ by Chae et al. in \cite{Chae-Constantin-Cordoba-Gancedo-Wu:gsqg-singular-velocities}.

 To get a better understanding of the behaviour of solutions of these interface problems, several numerical experiments have been performed. In \cite{Cordoba-Fontelos-Mancho-Rodrigo:evidence-singularities-contour-dynamics}, C\'ordoba et al. studied the problem of the evolution of two patches for a range of $\alpha$. Their simulations suggest an asymptotically self-similar singular scenario in which the distance between both patches goes to zero in finite time while simultaneously the curvature of the boundaries blows up. Scott and Dritschel  \cite{Scott-Dritschel:self-similar-sqg}, based on numerical simulations, suggest that an elliptical patch with a big aspect ratio between its axes may develop a self-similar singularity with an explosive growth of the curvature in the case $\al = 1$.  Recently (\cite{Castro-Cordoba-GomezSerrano-MartinZamora:remarks-geometric-properties-sqg}) it has been shown that elliptical patches are not rotating solutions for $\al > 0$, as opposed to the limiting case $\al \rightarrow 0$, - for which they are - and by means of a rigorous computer-assisted proof the existence of convex solutions that lose their convexity in finite time has been established. In a paper by Scott \cite{Scott:scenario-singularity-quasigeostrophic}, it was already pointed out that small perturbations of thin strips may lead to a self similar cascade of instabilities, leading to a possible arc chord blow up. Gancedo and Strain \cite{Gancedo-Strain:absence-splash-muskat-SQG} proved that in fact, no splash singularity can be formed, i.e., two interfaces can not collapse in a point, if the interfaces remain smooth.


 The evolution equation for the interface of an $\alpha-$ patch, which we parametrize as a $2 \pi$ periodic curve $z(x)$, can be written as
 \begin{align}
\label{Ecuacion-alpha-patch}
 \partial_t z(x,t) = -(\theta_2 - \theta_1)C(\al) \int_{0} ^{ 2 \pi} \frac{ \partial_x z (x,t) - \partial_x z(x-y,t) }{ \vert z(x,t) - z(x-y, t) \vert^{\alpha}} dy,
 \end{align}

 since we can add terms in the tangential direction without changing the evolution of the patch.

The normalizing constant $C(\al)$ is given by:

\begin{align*}
 C(\al) = \frac{1}{2\pi} \frac{\Gamma\left(\frac{\al}{2}\right)}{2^{1-\al}\Gamma\left(\frac{2-\al}{2}\right)}.
\end{align*}






 The analogous problem for the vorticity formulation of 2D Euler $(\alpha \to 0)$ is better understood. The global existence and uniqueness of weak solutions of the 2D Euler in vorticity formulation is due to Yudovich \cite{Yudovich:Nonstationary-ideal-incompressible}. Regularity preservation for $\mathcal{C}^{1, \gamma}$ patches was obtained by Chemin using techniques from paradifferential calculus in \cite{Chemin:persistance-structures-fluides-incompressibles}. Another proof of that result, which highlights the extra cancellation on semi spheres of even kernels, can be found in \cite{Bertozzi-Constantin:global-regularity-vortex-patches} by Bertozzi and Constantin. Serfati, in \cite{Serfati:preuve-directe-existence-globale-patches} provided another one, giving a fuller characterization of the velocity gradient's regularity.

  In recent years, Denisov has studied the process of merging for the vortex patch problem. This is the scenario showed by the numerics of \cite{Cordoba-Fontelos-Mancho-Rodrigo:evidence-singularities-contour-dynamics} for the $\alpha$-patch. However, for the vortex patch problem, the collapse in a point can not happen in finite time, the distance between the two patches can decay at most as fast as a double exponential. Denisov proves in  \cite{Denisov:sharp-corner-euler-patches} that this bound is sharp if one is allowed to modify slightly the velocity by superimposing a smooth background incompressible flow.

However, there is a family of solutions that evolve by rotating with constant angular velocity around its center of mass. These solutions are known as V-states. Deem and Zabusky were the first to compute them numerically \cite{Deem-Zabusky:vortex-waves-stationary}, and later other authors have improved the methods and numerically computed a bigger class (see \cite{Wu-Overman-Zabusky:steady-state-Euler-2d,Elcrat-Fornberg-Miller:stability-vortices-cylinder,LuzzattoFegiz-Williamson:efficient-numerical-method-steady-uniform-vortices,Saffman-Szeto:equilibrium-shapes-equal-uniform-vortices} for a small sample of them).

In the case $\alpha = 0$, Burbea \cite{Burbea:motions-vortex-patches} outlined a proof of the existence of V-states by means of a conformal mapping and bifurcation theory. A fully rigorous proof was given by Hmidi, Mateu and Verdera in \cite{Hmidi-Mateu-Verdera:rotating-vortex-patch}. They also showed that the family of V-states has $C^{\infty}$ boundary regularity and is convex. In another paper \cite{Hmidi-Mateu-Verdera:rotating-doubly-connected-vortices}, they studied the V-state existence for the case of doubly connected domains.

In a very recent preprint, Hassainia and Hmidi have worked in extending the ideas of the aforementioned papers to the case $0 < \alpha \leq 1$ \cite{Hassainia-Hmidi:v-states-generalized-sqg}. They are able to prove the existence of convex V-states with $C^{k}$ boundary regularity for the case $0 < \alpha < 1$. The possibility of $C^{\infty}$ regularity and the existence for the case $1 \leq \alpha < 2$ are left open.

Motivated by their work, we have attempted to fill the gap. Precisely, in this paper we are able to prove existence and $C^{\infty}$ regularity of convex global rotating solutions for the remaining open cases of $\alpha$. For the existence part, the key ingredient in our proof is a careful choice of the spaces in which we apply the Crandall-Rabinowitz theorem in a similar spirit as in the previous papers \cite{Burbea:motions-vortex-patches,Hassainia-Hmidi:v-states-generalized-sqg,Hmidi-Mateu-Verdera:rotating-vortex-patch,Hmidi-Mateu-Verdera:rotating-doubly-connected-vortices}. Concerning the regularity, one has to invert the most singular operator onto the less singular one to be able to bootstrap.

From now on, we will assume that $\theta_2 - \theta_1 = 1$.

\subsection{Contour equations for the rotating solutions}

Let $z(x,t) = (z_1(x,t),z_2(x,t))$ be the interface of the patch. Since our results will be concerned with patches that are close to the disk, we can assume that the patch is star-shaped and therefore it can be parametrized as $(R(x,t)\cos(x),R(x,t)\sin(x))$. In order to obtain the equations for $R(x,t)$, we will start writing up the equation for $z(x,t)$, and then substitute for the star-shaped ansatz.

Let us assume that $z(x,t)$ rotates with frequency $\Omega$ counterclockwise. Thus

\begin{align*}
z_t(x,t) = \Omega z^{\perp}(x,t),
\end{align*}

where for every $v = (v_1, v_2)$, $v^{\perp}$ is defined as $(-v_2,v_1)$. The equations a V-state satisfies are

\begin{align*}
z_t(x,t) \cdot n & = u(z(x,t),t) \cdot n \\
\Omega \langle z^{\perp}(x,t), z_{x}^{\perp}(x,t) \rangle  = \Omega \langle z(x,t), z_{x}(x,t) \rangle & = \langle u(z(x,t),t), z_{x}^{\perp}(x,t) \rangle.
\end{align*}

Here, $n$ is the unitary normal vector and the tangential component of the velocity does not change the shape of the curve. The question of finding a rotating global solution patch of the generalized quasi-geostrophic equation is reduced to find a zero of $F(\Omega,R)$, where

\begin{align}
\label{funcionalvstates}
F(\Omega,R) = \Omega R'(x) - \sum_{i=1}^{3} F_{i}(R),
\end{align}

and the $F_{i}$ are

\begin{align*}
F_1(R)=&\frac{1}{R(x)}C(\al)\int\frac{\sin(x-y)}{\left(\left(R(x)-R(y)\right)^2+4R(x)R(y)\sin^2\left(\frac{x-y}{2}\right)\right)^\frac{\al}{2}}
\left(R(x)R(y)+R'(x)R'(y)\right)dy,\\
F_2(R)=&C(\al)\int\frac{\cos(x-y)}{\left(\left(R(x)-R(y)\right)^2+4R(x)R(y)\sin^2\left(\frac{x-y}{2}\right)\right)^\frac{\al}{2}}
\left(R'(y)-R'(x)\right)dy,\\
F_3(R)=&\frac{R'(x)}{R(x)}C(\al)\int\frac{\cos(x-y)}{\left(\left(R(x)-R(y)\right)^2+4R(x)R(y)\sin^2\left(\frac{x-y}{2}\right)\right)^\frac{\al}{2}}
\left(R(x)-R(y)\right)dy,
\end{align*}

and the above integrals are performed on the torus. For simplicity, from now on we will omit writing the domain of integration, which is always the torus.

\subsection{Functional spaces}

In our proofs, we will use the following spaces:

\begin{align*}
X^{k} = \left\{f \in H^{k}, f(x) = \sum_{j=1}^{\infty}a_{j} \cos(jx)\right\}, \quad
X^{k}_{m} = \left\{f \in H^{k}, f(x) = \sum_{j=1}^{\infty}a_{jm} \cos(jmx)\right\} \\
Y^{k} = \left\{f \in H^{k}, f(x) = \sum_{j=1}^{\infty}a_{j} \sin(jx)\right\}, \quad
Y^{k}_{m} = \left\{f \in H^{k}, f(x) = \sum_{j=1}^{\infty}a_{jm} \sin(jmx)\right\} \\
X^{k+\log} = \left\{f \in H^{k}, f(x) = \sum_{j=1}^{\infty}a_{j} \cos(jx), \left\|\int_{\mathbb{T}} \frac{\pa^{k}f(x-y)-\pa^{k}f(x)}{\left|\sin\left(\frac{y}{2}\right)\right|}dy\right\|_{L^2(x)} < \infty\right\}, \quad k \in \mathbb{Z} \\
X^{k+\log}_{m} = \left\{f \in H^{k}, f(x) = \sum_{j=1}^{\infty}a_{jm} \cos(jmx), \left\|\int_{\mathbb{T}} \frac{\pa^{k}f(x-y)-\pa^{k}f(x)}{\left|\sin\left(\frac{y}{2}\right)\right|}dy\right\|_{L^2(x)} < \infty\right\}, \quad k \in \mathbb{Z}.
\end{align*}

The norm is given in the last two cases by the sum of the $H^{k}$-norm and the additional finite integral in the definition, and in the other four by the $H^{k}$ norm. We give an alternative characterization of the $X^{k+\log}$ spaces.  This will be useful for the spectral study.

\begin{prop}
\label{alternativexlog}
\begin{align*}
f \in X^{k+\log} \Leftrightarrow f \in X^{k}, |a_1|^{2} + \sum_{j=2}^{\infty}|a_j|^{2}|j|^{2k}(1+\log(j))^{2} < \infty,
\end{align*}
where \begin{align*}
f = \sum_{j=1}^{\infty} a_{j} \cos(jx)
\end{align*}
\end{prop}
\begin{proof}
$\Rightarrow:$
%

By virtue of Lemmas \ref{LemmaICIS} and \ref{lemmagrowthomega}:

\begin{align*}
|a_1|^{2} + \sum_{j=2}^{\infty}|a_j|^{2}|j|^{2k}(1+\log(j))^{2} \leq C\left(|a_1|^{2} + \sum_{j=2}^{\infty}|a_j|^{2}|j|^{2k}(\log(j))^{2}\right)  \leq C\left\|\int \frac{\pa^{k} f(x-y)- \pa^{k} f(x)}{\left|\sin\left(\frac{y}{2}\right)\right|}dy\right\|_{L^{2}}^{2} < \infty
\end{align*}

$\Leftarrow:$

Since $f \in X^{k}$, it can be written as a Fourier series. Let those coefficients be $a_k$. By Lemmas \ref{LemmaICIS} and
\ref{lemmagrowthomega}:

\begin{align*}
 \left\|\int \frac{\pa^{k} f(x-y)- \pa^{k} f(x)}{\left|\sin\left(\frac{y}{2}\right)\right|}dy\right\|_{L^{2}}^{2} \leq C \left(|a_1|^{2} + \sum_{j=2}^{\infty}|a_j|^{2}|j|^{2k} (\log(j))^{2}\right) \leq  C \left(|a_1|^{2} + \sum_{j=2}^{\infty}|a_j|^{2}|j|^{2k} (1+\log(j))^{2}\right) < \infty.
\end{align*}

\end{proof}

\begin{rem}
There is a substantial difference between the spaces $X^{k+\log}$ and the spaces $B^{s}$ and $B^{s-1}_{Log}$ that were proposed as candidates in \cite{Hassainia-Hmidi:v-states-generalized-sqg}. Even though the Fourier multiplier scaling is correct, the $l^{1}$-summability condition and the definition via Fourier series does not allow to estimate the nonlinear terms in an easy way. Finding an alternative characterization in physical space removes this obstacle. Moreover, the choice of $X^{k+\log}$ as a space that wins a logarithm of a derivative instead of finding a space that loses a logarithm of a derivative (as suggested in \cite{Hassainia-Hmidi:v-states-generalized-sqg}) alleviates the heavy computations. We make no claim that the proposed spaces $B^{s},B^{s-1}_{Log}$ might not work.
\end{rem}

\subsection{Theorems and outline of the proofs}

 The paper is organized as follows:

In Section \ref{sectionexistencealpha1}, we prove the following theorem:

\begin{theorem}
\label{teoremaexistenciaalpha1}
 Let $k\geq 3, m \in \mathbb{N}, m \geq 2$ and let

 \begin{align*} \Omega_m = -\frac{2}{\pi}\sum_{k=2}^{m}\frac{1}{2k-1}.\end{align*}

Then, there exists a family of $m$-fold solutions $(\Omega,R), R(x)-1 \in X^{k+\log}_{m}$ of the equation \eqref{funcionalvstates} with $\al = 1$ that bifurcate from the disk at $\Omega = \Omega_m$.
\end{theorem}

Section \ref{Sectionexistencealphamayor1} is devoted to prove

\begin{theorem}
\label{teoremaexistenciaalphamayor1}
 Let $k\geq 3, m \in \mathbb{N}, m \geq 2, 1 < \alpha < 2$ and let

 \begin{align*} \Omega_m = -2^{\al-1} \frac{\Gamma\left(1-\al\right)}{\left(\Gamma\left(1-\frac{\al}{2}\right)\right)^{2}}\left(\frac{\Gamma\left(1+\frac{\al}{2}\right)}{\Gamma\left(2-\frac{\al}{2}\right)} - \frac{\Gamma\left(m+\frac{\al}{2}\right)}{\Gamma\left(1+m-\frac{\al}{2}\right)}\right).\end{align*} Then, there exists a family of $m$-fold solutions $(\Omega,R), R(x)-1 \in X^{k}_{m}$ of the equation \eqref{funcionalvstates} with $1 < \al < 2$ that bifurcate from the disk at $\Omega = \Omega_m$.
\end{theorem}

Both proofs are carried out by means of a combination of a Crandall-Rabinowitz's theorem and a priori estimates.

\begin{rem}
 We remark that there is a lot of room for improvement of the regularity, and the choice $k \geq 3$ is far from being optimal. Here we are only interested on finding one such $k$ that shows the theorem, and not the sharpest one in terms of regularity.
\end{rem}

In the final section we deal with the regularity of the boundary and its convexity. We are able to show the following theorem:

\begin{theorem}
\label{teoremaregularidadconvexidad}
 Let $0 < \alpha < 2$. Let $R(x)$ be an $m$-fold solution of \eqref{funcionalvstates} which is close to the disk. Then, $R(x)$ belongs to $C^{\infty}$ and it parametrizes a convex patch.
\end{theorem}

To do so, we will invert a singular integral operator to gain regularity and bootstrap. It is not necessary to derive an explicit formula as in \cite{Hmidi-Mateu-Verdera:rotating-vortex-patch} using the structure of the kernel. This agrees with the discussion in \cite[Remark 3]{Hassainia-Hmidi:v-states-generalized-sqg}.

\section{Existence in the case $\alpha=1$}
\label{sectionexistencealpha1}
This section is devoted to show Theorem \ref{teoremaexistenciaalpha1}.

\begin{proof}
 The proof will be divided into 6 steps. These steps correspond to check the hypotheses of the Crandall-Rabinowitz theorem \cite{Crandall-Rabinowitz:bifurcation-simple-eigenvalues} for $$F(\Omega,R)=\Omega R'-\sum_{i=1}^3F_i(R),$$
where
\begin{align*}
F_1(R)=&\frac{1}{2\pi R(x)}\int\frac{\sin(x-y)}{\left(\left(R(x)-R(y)\right)^2+4R(x)R(y)\sin^2\left(\frac{x-y}{2}\right)\right)^\frac{1}{2}}
\left(R(x)R(y)+R'(x)R'(y)\right)dy,\\
F_2(R)=&\frac{1}{2\pi }\int\frac{\cos(x-y)}{\left(\left(R(x)-R(y)\right)^2+4R(x)R(y)\sin^2\left(\frac{x-y}{2}\right)\right)^\frac{1}{2}}
\left(R'(y)-R'(x)\right)dy,\\
F_3(R)=&\frac{R'(x)}{2\pi R(x) }\int\frac{\cos(x-y)}{\left(\left(R(x)-R(y)\right)^2+4R(x)R(y)\sin^2\left(\frac{x-y}{2}\right)\right)^\frac{1}{2}}
\left(R(x)-R(y)\right)dy,
\end{align*}
and they are the following:
\begin{enumerate}
\item The functional $F$ satisfies $$F(\Omega,R)\,:\, \R\times \{1+V^r\}\mapsto Y^{k-1},$$ where $V^r$ is the open neighborhood of 0
$$V^r=\{ f\in X^{k+\log}\,:\, ||f||_{X^{k+\log}}<r\},$$
for all $0<r<1$ and $k\geq 3$.
\item $F(\Omega,1) = 0$ for every $\Omega$.
\item The partial derivatives $F_{\Omega}$, $F_{R}$ and $F_{R\Omega}$ exist and are continuous.
\item Ker($\mathcal{F}$) and $Y^{k-1}$/Range($\mathcal{F}$) are one-dimensional, where $\mathcal{F}$ is the linearized operator around the disk $R = 1$ at $\Omega = \Omega_m$.
\item $F_{\Omega R}(\Omega_m,1)(h_0) \not \in$ Range($\mathcal{F}$), where Ker$(\mathcal{F}) = \langle h_0 \rangle$.
\item Step 1 can be applied to the spaces $X^{k+\log}_{m}$ and $Y^{k-1}_{m}$ instead of $X^{k+\log}$ and $Y^{k-1}$.
\end{enumerate}

\subsection{Step 1}

In order to prove that
$$F(\Omega,R)\,:\, \R\times \{1+V^r\}\mapsto Y^{k-1},$$ where $V^r$ is the open neighborhood of 0
$$V^r=\{ f\in X^{k+\log}\,:\, ||f||_{X^{k+\log}}<r\},$$
for all $0<r<1$ and $k\geq 3$,
we will deal with the most singular terms. For example we will not give details about the bound on the $L^2-$ norm of $F(R)$ and we focus on the following proposition:

\begin{prop}\label{prop1}
Let $0 < r < 1$, $k \geq 3$.  Then
 $$F(\Omega,R): \mathbb{R} \times \{1+V^r\} \mapsto Y^{k-1}$$
\end{prop}
\begin{proof}
In order to show this proposition we will use the following decomposition
\begin{align*}
\frac{1}{\left(\left(R(x)-R(y)\right)^2+4R(x)R(y)\sin^2\left(\frac{x-y}{2}\right)\right)^\frac{1}{2}}
= K_{S}(x,y)+\frac{1}{\left(R(x)^2+R'(x)^2\right)^\frac{1}{2}}\frac{1}{2\left|\sin\left(\frac{x-y}{2}\right)\right|},
\end{align*}
where the kernel
\begin{align*}
K_S(x,y)\equiv \frac{1}{\left(\left(R(x)-R(y)\right)^2+4R(x)R(y)\sin^2\left(\frac{x-y}{2}\right)\right)^\frac{1}{2}}-
\frac{1}{\left(R(x)^2+R'(x)^2\right)^\frac{1}{2}}\frac{1}{2\left|\sin\left(\frac{x-y}{2}\right)\right|},
\end{align*}
is not singular at $x=y$.

Since we can write
\begin{align*}
\cos(x)=&1-2\sin^2\left(\frac{x}{2}\right),
\end{align*}
we have that
\begin{align*}
 &\frac{\cos(x-y)}{\left(\left(R(x)-R(y)\right)^2+4R(x)R(y)\sin^2\left(\frac{x-y}{2}\right)\right)^\frac{1}{2}}\\&=
\cos(x-y)K_S(x,y)+  \frac{1}{\left(R(x)^2+R'(x)^2\right)^\frac{1}{2}}\frac{1}{2\left|\sin\left(\frac{x-y}{2}\right)\right|}
-\frac{\left|\sin\left(\frac{x-y}{2}\right)\right|}{\left(R(x)^2+R'(x)^2\right)^\frac{1}{2}}
\end{align*}
Let's bound $F_1$. We will split $F_1$ into two terms,
\begin{align*}
F_1=&\frac{1}{2\pi R(x)}\int\frac{\sin(x-y)}{\left(\left(R(x)-R(y)\right)^2+4R(x)R(y)\sin^2\left(\frac{x-y}{2}\right)\right)^\frac{1}{2}}
R(x)R(y)dy\\
&+\frac{1}{2\pi R(x)}\int\frac{\sin(x-y)}{\left(\left(R(x)-R(y)\right)^2+4R(x)R(y)\sin^2\left(\frac{x-y}{2}\right)\right)^\frac{1}{2}}
R'(x)R'(y)dy\\
\equiv & F_{11}+F_{12}
\end{align*}
and we will focus on $F_{12}$ since it is the most singular one.

Making the change of variable  $x-y \mapsto y$, taking $\pa^{k-1}$ derivatives with respect to $x$  and changing back again to $y\mapsto x-y$ yields
\begin{align*}
\pa^{k-1}F_{12}=&\frac{(-1) \pa^{k-1}R(x)}{2\pi R(x)^{2}}\int\frac{\sin(x-y)}{\left(\left(R(x)-R(y)\right)^2+4R(x)R(y)\sin^2\left(\frac{x-y}{2}\right)\right)^\frac{1}{2}}
R'(x)R'(y)dy\\
+&\frac{1}{2\pi R(x)}\int\frac{\sin(x-y)}{\left(\left(R(x)-R(y)\right)^2+4R(x)R(y)\sin^2\left(\frac{x-y}{2}\right)\right)^\frac{1}{2}}
\left(R(x)\pa^kR(y)+\pa^kR(x)R(y)\right)dy\\
&-\frac{1}{2\pi R(x)}\int\frac{\sin(x-y)R'(x)R'(y)}{\left(\left(R(x)-R(y)\right)^2+4R(x)R(y)\sin^2\left(\frac{x-y}{2}\right)\right)^\frac{3}{2}}
\\ & \times \left((R(x)-R(y))\left(\pa^{k-1}R(x)-\pa^{k-1}R(y)\right)+2\left(\pa^{k-1}R(x)R(y)+R(x)\pa^{k-1}R(y)\right)\sin^2\left(\frac{x-y}{2}\right)\right)dy
\\&+ \text{l.o.t},
\end{align*}
where l.o.t stands for lower order terms.

\begin{defi}
 Let $k \in \mathbb{R}$. We denote by $\mathcal{H}_{k}$ the set of functions $f(x,y)$ that satisfy the following estimates:

\begin{align*}
\sup_{x\in \T} \left\|\frac{f(x,\cdot)}{\sin(\cdot)^{k}}\right\|_{L^1(\T)}\leq C, \quad  \sup_{y \in \T}\left\|\frac{f(\cdot,y)}{\sin(y)^{k}}\right\|_{L^1(\T)}\leq C,
\end{align*}

where $C$ is a constant. Its dependence will be clear on the context.

\end{defi}

 Since the kernel
$$a(x,y)\equiv \frac{\sin(x-y)}{\left(\left(R(x)-R(y)\right)^2+4R(x)R(y)\sin^2\left(\frac{x-y}{2}\right)\right)^\frac{1}{2}}$$
belongs to $\mathcal{H}_0$
we can apply Young's inequality to prove that
$$||\pa^{k-1} F_{1}||_{L^2(\T)} \leq C\left(||R||_{X^{k+\log}},r\right).$$
Next we bound $\pa^{k-1} F_2$. Making the change of variable  $x-y \mapsto y$, taking $\pa^{k-1}$ derivatives with respect to $x$  and changing again to $y\mapsto x-y$ yields
\begin{align*}
\pa^{k-1} F_{2}& =\frac{1}{2\pi }\int\frac{\cos(x-y)}{\left(\left(R(x)-R(y)\right)^2+4R(x)R(y)\sin^2\left(\frac{x-y}{2}\right)\right)^\frac{1}{2}}
\left(\pa^k R(y)-\pa^k R(x)\right)dy\\
& -\frac{1}{2\pi }\int\frac{\cos(x-y)}{\left(\left(R(x)-R(y)\right)^2+4R(x)R(y)\sin^2\left(\frac{x-y}{2}\right)\right)^\frac{3}{2}}\left(\pa^{k-1} R(y)-\pa^{k-1} R(x)\right) \\
& \times \left((R(x)-R(y))(R'(x)-R'(y)) + 2(R(x)R'(y) + R'(x)R(y))\sin^{2}\left(\frac{x-y}{2}\right)\right)dy\\
&+ \text{l.o.t.}
\end{align*}
We will split the first term as follows
\begin{align*}
&\frac{1}{2\pi }\int \cos(x-y)K_{S}(x,y)\left(\pa^k R(y)-\pa^kR(x)\right)dy\\
&-\frac{1}{2\pi}\int \frac{\left|\sin\left(\frac{x-y}{2}\right)\right|}{\left(R(x)^2+R'(x)^2\right)^\frac{1}{2}}\left(\pa^k R(y)-\pa^k R(x)\right)dy\\
&+\frac{1}{4\pi \left(R(x)^2+R'(x)^2\right)^\frac{1}{2}} \int \frac{(\pa^k R(y)-\pa^k R(x))}{\left|\sin\left(\frac{x-y}{2}\right)\right|}dy.
\end{align*}
Therefore, since $K_S \in \mathcal{H}_0$
and because of the definition of the space $X^{k+\log}$ we have that
\begin{align*}
||\pa^{k-1} F_2||_{L^2(\T)}\leq C\left(||R||_{X^{k+\log}},r\right),
\end{align*}

since the first integral is bounded via Young's inequality, the second one is of lower order and the third one is precisely part of the definition of the $X^{k+\log}$ spaces.

To estimate the second, we subtract and add the following term:

\begin{align*}
\frac{1}{2\pi }\int\frac{\cos(x-y)}{2\left|\sin\left(\frac{x-y}{2}\right)\right|}\frac{R'(x)R''(x)+R'(x)R(x)}{((R(x))^2+(R'(x))^2)^{3/2}}\left(\pa^{k-1} R(y)-\pa^{k-1} R(x)\right) dy,
\end{align*}

which is clearly in $L^{2}$, so that the resulting kernel coming from the difference, which is acting on $(\pa^{k-1} R(y)-\pa^{k-1}R(x))$, belongs to $\mathcal{H}_{0}$. By using Young's inequality again, we get the desired bound. Finally, the bound for $F_3$ is  easier to get than the one for $F_2$.
\end{proof}
\subsection{Step 2}

If we substitute $R = 1$ in the $F_{i}$, the only term which is not immediately 0 is the first part of $F_{1}$. Therefore, we are left to show that

\begin{align*}
  \int_{-\pi}^{\pi} \frac{\sin(y)}{\left(4\sin^{2}\left(\frac{y}{2}\right)\right)^{1/2}} dy = 0,\\
\end{align*}

but this is automatically true since the integrand is odd. Thus, $F(\Omega,1)=0$ for all $\Omega\in \R$.

\subsection{Step 3}
We need to prove the existence and the continuity of the Gateaux derivatives $\pa_\Omega F(\Omega,R)$, $\pa_R F(\Omega,R)$ and $\pa_{\Omega,R}F(\Omega,R)$. The most difficult part is to show the existence and continuity of $\pa_R F_i(R)$ for $i=1,2,3$ since the dependence on $\Omega$ is linear and the rest follows easily.

\begin{lemma} \label{gatoderivada}For all $R-1\in V^r$ and for all $h\in X^{k+\log}$ such that $||h||_{X^{k+\log}}=1$ we have that
$$\lim_{t \to 0}\frac{F_i(R+th)-F_i(R)}{t}= D_i[R]h\quad \text{in $Y^{k-1}$},$$
where
\begin{align*}
D_1[R] h =& -\frac{h(x)}{2\pi R(x)^2}\int \frac{\sin(x-y)}{\left((R(x)-R(y))^2+4R(x)R(y)\sin^2\left(\frac{x-y}{2}\right)\right)^\frac{1}{2}} \left(R(x)R(y)+R'(x)R'(y)\right)dy\\
& +\frac{1}{2\pi R(x)}\int \frac{\sin(x-y)}{\left((R(x)-R(y))^2+4R(x)R(y)\sin^2\left(\frac{x-y}{2}\right)\right)^\frac{1}{2}}\\&\times(h(x)R(y)+h(y)R(x)+(h'(x)R'(y)+h'(y)R'(x)))dy\\
&-\frac{1}{2\pi R(x)}\int \frac{\sin(x-y)(R(x)R(y)+R'(x)R'(y))}{\left((R(x)-R(y))^2+4R(x)R(y)\sin^2\left(\frac{x-y}{2}\right)\right)^\frac{3}{2}}\\
&\times\left((R(x)-R(y))(h(x)-h(y))+2(h(x)R(y)+h(y)R(x))\sin^2\left(\frac{x-y}{2}\right)\right)dy\\
D_2[R] h = & \frac{1}{2\pi}\int\frac{\cos(x-y)}{\left((R(x)-R(y))^2+4R(x)R(y)\sin^2\left(\frac{x-y}{2}\right)\right)^\frac{1}{2}}\left(h'(y)-h'(x)\right)dy\\
&-\frac{1}{2\pi}\int\frac{\cos(x-y)(R'(y)-R'(x))}{\left((R(x)-R(y))^2+
4R(x)R(y)\sin\left(\frac{x-y}{2}\right)\right)^\frac{3}{2}}\\ & \times \left((R(x)-R(y))(h(x)-h(y))+2(h(x)R(y)+h(y)R(x))\sin^2\left(\frac{x-y}{2}\right)\right)dy\\
D_3[R]h= & \frac{h'(x)}{2\pi R(x)}\int\frac{\cos(x-y)}{\left(\left(R(x)-R(y)\right)^2+4R(x)R(y)\sin^2\left(\frac{x-y}{2}\right)\right)^\frac{1}{2}}(R(x)-R(y))dy\\
&-\frac{R'(x) h(x)}{2\pi R(x)^2}\int\frac{\cos(x-y)}{\left(\left(R(x)-R(y)\right)^2+4R(x)R(y)\sin^2\left(\frac{x-y}{2}\right)\right)^\frac{1}{2}}
(R(x)-R(y)) dy \\
&+\frac{R'(x)}{2\pi R(x)}\int\frac{\cos(x-y)}{\left(\left(R(x)-R(y)\right)^2+4R(x)R(y)\sin^2\left(\frac{x-y}{2}\right)\right)^\frac{1}{2}}(h(x)-h(y))dy\\
&-\frac{R'(x)}{2\pi R(x)}\int \frac{\cos(x-y)(R(x)-R(y))}{\left((R(x)-R(y))^2+
4R(x)R(y)\sin\left(\frac{x-y}{2}\right)\right)^\frac{3}{2}}\\ & \times \left((R(x)-R(y))(h(x)-h(y))+2(h(x)R(y)+h(y)R(x))\sin^2\left(\frac{x-y}{2}\right)\right)dy.
\end{align*}

Moreover, $D_i[R] h$ are continuous in $R$.
\end{lemma}
\begin{proof}

We will focus on the term $F_2(R)$ since it is the one that requires a special treatment with respect to the case $\al < 1$ discussed in \cite{Hassainia-Hmidi:v-states-generalized-sqg}. We need to prove that
$$\lim_{t\to 0} \left|\left| \frac{F_2(R+th)-F_2(R)}{t}-D_2[R]h\right|\right|_{H^{k-1}}=0.$$
In order to do it we will use the following notation. For a general function $f(x)$ we will write
\begin{align*}
\Delta f  = &  f(x)-f(y),\\
f=&f(x),  \\
\overline{f}=& f(y),
\end{align*}
and also we will use the following abbreviations
\begin{align*}
DR= & \Dr^2+4 R\oR \sin^2\left(\frac{x-y}{2}\right),\\ DRh=& (\Dr+t\Dh)^2+4 (R+th)(\oR+t\oh) \sin^2\left(\frac{x-y}{2}\right).
\end{align*}

We decompose the term inside the $H^{k-1}$-norm in the following way
\begin{align*}
&\frac{F_2(R+th)-F_2(R)}{t}-D_2[R]h= \\
& \frac{1}{2\pi t}\int \cos(x-y)(R'(x)-R'(y))\\
&\times \left(\frac{1}{\left(\left(\Dr+t\Dh \right)^2+4(R+th)(\oR+t\oh)\sin^2\left(\frac{x-y}{2}\right)\right)^\frac{1}{2}}-
\frac{1}{\left(\left(\Dr\right)^2+4R\oR\sin^2\left(\frac{x-y}{2}\right)\right)^\frac{1}{2}}\right.\\
&\left. + t \frac{\Dr\Dh+2(h\oR+\oh R)\sin^2\left(\frac{x-y}{2}\right)}{\left(\left(\Dr\right)^2+4R\oR\sin^2\left(\frac{x-y}{2}\right)\right)^\frac{3}{2}}\right)dy\\
&+\frac{1}{2\pi}\int \cos(x-y)(h'(x)-h'(y))\\
&\times\left(\frac{1}{\left(\left(\Dr+t\Dh \right)^2+4(R+th)(\oR+t\oh)\sin^2\left(\frac{x-y}{2}\right)\right)^\frac{1}{2}}
-\frac{1}{\left(\left(\Dr\right)^2+4R\oR\sin^2\left(\frac{x-y}{2}\right)\right)^\frac{1}{2}} \right)dy.
&\equiv I_1+I_2
\end{align*}
We will deal with $I_2$ first,
\begin{align*}
I_2=&\frac{1}{2\pi}\int \cos(x-y)(h'(x)-h'(y))\\
&\times\left(\frac{1}{\left(\left(\Dr+t\Dh \right)^2+4(R+th)(\oR+t\oh)\sin^2\left(\frac{x-y}{2}\right)\right)^\frac{1}{2}}
-\frac{1}{\left(\left(\Dr\right)^2+4R\oR\sin^2\left(\frac{x-y}{2}\right)\right)^\frac{1}{2}} \right)dy\\
=&\frac{1}{2\pi}\int \cos(x-y)(h'(x)-h'(y))\left(\frac{DR-DRh}{DR^\frac{1}{2}DRh^\frac{1}{2}\left(DR^\frac{1}{2}+DRh^\frac{1}{2}\right)}\right)dy\\
=&\frac{1}{2\pi}\int \cos(x-y)(h'(x)-h'(y))
\left(\frac{-2t\Dr\Dh-t^2\Dh^2-4(t(h\oR+\oh R)+t^2h\oh) \sin^2\left(\frac{x-y}{2}\right)}{DR^\frac{1}{2}DRh^\frac{1}{2}\left(DR^\frac{1}{2}+DRh^\frac{1}{2}\right)}\right)dy\\
&\equiv I_{21}+I_{22}+I_{23}+I_{24}.
\end{align*}
These four terms can be treated in a similar way. We just give some details about $I_{21}$,
\begin{align*}
I_{21}=-\frac{t}{\pi}\int \cos(x-y)(h'(x)-h'(y))
\left(\frac{\Dr\Dh}{DR^\frac{1}{2}DRh^\frac{1}{2}\left(DR^\frac{1}{2}+DRh^\frac{1}{2}\right)}\right)dy.
\end{align*}
Now we will take $k-1$ derivatives with respect to $x$. As before, we first make the change $x-y\mapsto y$, then we apply $\pa_x^{k-1}$ and finally we make the change $y\mapsto x-y$, thus
\begin{align*}
\pa^{k-1} I_{21}=& -\frac{t}{\pi}\int \cos(x-y)(\pa^kh(x)-\pa^k h(y))\left(\frac{\Dr\Dh}{DR^\frac{1}{2}DRh^\frac{1}{2}\left(DR^\frac{1}{2}+DRh^\frac{1}{2}\right)}\right)dy\\
&+ \text{  l.o.t}
\end{align*}
Now we can decompose the kernel in the previous integral in the following way
\begin{align*}
&\frac{\Dr\Dh}{DR^\frac{1}{2}DRh^\frac{1}{2}\left(DR^\frac{1}{2}+DRh^\frac{1}{2}\right)}=\tilde{K}_{S}\\&
+\frac{R'h'}{(R'^2+R^2)^\frac{1}{2}((R'+th')^2+(R+th)^2)^\frac{1}{2}((R'^2+R^2)^\frac{1}{2}+((R'+th')^2+(R+th)^2)^\frac{1}{2})}
\frac{1}{2\left|\sin\left(\frac{x-y}{2}\right)\right|},
\end{align*}
where we can check that, for $t$ small enough, $\tilde{K}_{S} \in \mathcal{H}_0$.
From this decomposition and for $t$ small enough we have that
\begin{align*}
||\pa^{k-1}I_{21}||_{L^2}\leq t C(||R||_{X^{k+\log}},r).
\end{align*}
Now let's deal with $I_1$. We proceed again with the same strategy as before in order to take $k-1$ derivatives, in such a way that

\begin{align*}
\pa^{k-1}I_1=&\frac{1}{2\pi t}\int \cos(x-y)(\pa^k R(x)-\pa^k R(y))\left(\frac{1}{DRh^\frac{1}{2}}-\frac{1}{DR^\frac{1}{2}}+t\frac{\Dr\Dh+2(h\oR+\oh R)\sin^2\left(\frac{x-y}{2}\right)}{DR^\frac{3}{2}}\right)dy\\
&+ \text{   l.o.t}.
\end{align*}
One can write
\begin{align*}
&\frac{1}{DRh^\frac{1}{2}}-\frac{1}{DR^\frac{1}{2}}+t\frac{\Dr\Dh+2(h\oR+\oh R)\sin^2\left(\frac{x-y}{2}\right)}{DR^\frac{3}{2}}\\
& = t\frac{-\Dr\Dh-2(h\oR+\oh R)\sin^2\left(\frac{x-y}{2}\right)}{DRh^\frac{1}{2} DRh^\frac{1}{2}\frac{DRh^\frac{1}{2}+DR^\frac{1}{2}}{2} }
+t\frac{\Dr\Dh+2(h\oR+\oh R)\sin^2\left(\frac{x-y}{2}\right)}{DR^\frac{3}{2}}\\
& = t \left(\Dr\Dh+2(h\oR+\oh R)\sin^2\left(\frac{x-y}{2}\right)\right)\left(\frac{1}{DR^\frac{3}{2}}-\frac{1}{DR^\frac{1}{2} DRh^\frac{1}{2}\frac{DRh^\frac{1}{2}+DR^\frac{1}{2}}{2}}\right),
\end{align*}
and also
\begin{align*}
\frac{1}{DR^\frac{3}{2}}-\frac{1}{DR^\frac{1}{2} DRh^\frac{1}{2}\frac{DRh^\frac{1}{2}+DR^\frac{1}{2}}{2}}=
\frac{1}{DR^\frac{1}{2}}\left(\frac{DRh^\frac{1}{2}\frac{DRh^\frac{1}{2}+DR^\frac{1}{2}}{2}-DR}{DR DRh^\frac{1}{2}\frac{DRh^\frac{1}{2}+DR^\frac{1}{2}}{2}}\right),
\end{align*}
where
\begin{align*}
DRh^\frac{1}{2}\frac{DRh^\frac{1}{2}+DR^\frac{1}{2}}{2}-DR& = \frac{1}{2}DRh-\frac{1}{2}DR+\frac{1}{2}(DRh^\frac{1}{2}-DR^\frac{1}{2})DR^\frac{1}{2}\\
&= \frac{1}{2}(DRh-DR)+\frac{DR^\frac{1}{2}}{2(DRh^\frac{1}{2}+DR^\frac{1}{2})}(DRh-DR).
\end{align*}

From these formulas it is easy to see that
\begin{align*}
&\frac{1}{DRh^\frac{1}{2}}-\frac{1}{DR^\frac{1}{2}}+t\frac{\Dr\Dh+2(h\oR+\oh R)\sin^2\left(\frac{x-y}{2}\right)}{DR^\frac{3}{2}}=t^2K[R,h,t]
\end{align*}
The $L^2$-bound for the term coming from the kernel $K[R,h,t]$ can be performed in a similar way as we did before.

To prove the continuity of $D_2[R] h$ we notice that, for $R$, $r$ in $V^r$ we have that

\begin{align*}
\frac{1}{DR^\frac{1}{2}}-\frac{1}{Dr^\frac{1}{2}}&=\frac{\Delta r^2 +4r\overline{r}\sin^2\left(\frac{x-y}{2}\right)-\Dr^2-4R\oR\sin^2\left(\frac{x-y}{2}\right)}
{DR^\frac{1}{2}Dr^\frac{1}{2}\left(DR^\frac{1}{2}+Dr^\frac{1}{2}\right)}\\
&=\frac{(\Delta r+\Dr)\Delta(r-R)+ 4\left((r-R)\overline{r}+(\overline{r}-\overline{R})R\right)\sin^2\left(\frac{x-y}{2}\right)}
{DR^\frac{1}{2}Dr^\frac{1}{2}\left(DR^\frac{1}{2}+Dr^\frac{1}{2}\right)}
\end{align*}
From this formula we obtain the estimate
\begin{align*}
||D_2[R] h -D_2[r] h ||_{H^{k-1}}\leq C\left(|| R||_{H^{k+\log}},|| r||_{H^{k+\log}},r\right)|| R-r||_{H^{k+\log}}
\end{align*}
what proves the continuity of $D_2[R]h$. The rest of the estimates can be performed in a similar fashion.
\end{proof}

\subsection{Step 4}

The calculations carried out in this subsection will be more general and include the full range $\al \geq 1$ instead of $\al = 1$. They will be used in the next section.

Before starting Step 4, we compute the linearization of $F$ around the disk ($R(x) \equiv 1$) in the direction $h(x)$. By taking $R=1$ in lemma \ref{gatoderivada} one sees that  this linearization is equal to

\begin{align*}
& \Omega h'(x)  - C(\al) \int \frac{\sin(y)(h(x-y)+h(x))}{\left(4\sin^2\left(\frac{y}{2}\right)\right)^{\al/2}}dy  +  \frac{\al}{2} C(\al) \int \frac{\sin(y)(h(x-y)+h(x))}{\left(4\sin^2\left(\frac{y}{2}\right)\right)^{\al/2}}dy + C(\al) \int \frac{\cos(y)(h'(x)-h'(x-y))}{\left(4\sin^2\left(\frac{y}{2}\right)\right)^{\al/2}}dy \\
& = \Omega h'(x) + \left(\frac{\al}{2}-1\right)C(\al) \int \frac{\sin(y)h(x-y)}{\left(4\sin^2\left(\frac{y}{2}\right)\right)^{\al/2}}dy  + C(\al) \int \frac{\cos(y)(h'(x)-h'(x-y))}{\left(4\sin^2\left(\frac{y}{2}\right)\right)^{\al/2}}dy, \\
\end{align*}

where we have used that

\begin{align*}
\int \frac{\sin(y)}{(4\sin^{2}\left(\frac{y}{2}\right))^{\al/2+1}} dy = 0. \\
\end{align*}

\begin{prop}
\label{representationOperator}
 Let $\al \geq 1$, and let $\displaystyle h = \sum_{j=1}^{\infty}a_k \cos(kx)$. Then,
\begin{align*}
\pa_{R}F(\Omega,1)(h) = -\sum_{k=1}^{\infty} a_k[k(\Omega-\Omega_k)] \sin(kx),
\end{align*}

where $\Omega_k$ is the dispersion set given by

\begin{align*}
 \left\{
\begin{array}{cc}
\displaystyle -2^{\al-1} \frac{\Gamma\left(1-\al\right)}{\left(\Gamma\left(1-\frac{\al}{2}\right)\right)^{2}}\left(\frac{\Gamma\left(1+\frac{\al}{2}\right)}{\Gamma\left(2-\frac{\al}{2}\right)} - \frac{\Gamma\left(k+\frac{\al}{2}\right)}{\Gamma\left(1+k-\frac{\al}{2}\right)}\right) & \alpha \neq 1 \\
 \displaystyle -\frac{2}{\pi} \sum_{j=2}^{k} \frac{1}{2j-1}& \alpha = 1
\end{array}
\right.
\end{align*}

\end{prop}
\begin{rem}
 The case $\al \leq 1$ was already covered by \cite{Hassainia-Hmidi:v-states-generalized-sqg}. Our proof goes the same way regardless of the value of $0 < \al < 2$. However, the expression of $\Omega_m$ is also valid in the range $\al > 1$. There is a slight discrepance on the sign, caused by the different choices of $\theta_2 - \theta_1$.
\end{rem}

\begin{proof} In order to calculate the critical rotating velocities, we shall look to each of the contributions to the $k$-th modes. Let $h \in X^{k+\log}_{m}$ or $X^{k}_{m}$, depending on the value of $\alpha$ be
 \begin{align*}
 h(x) = \sum_{j=1}^{\infty} a_j \cos(jx).
 \end{align*}

Then, the contribution of the derivative term $\Omega h'(x)$ to the $k-$th (sine) mode is given by $-k a_{k} \Omega$.

Let's look at the other terms' contribution. The first one will be:

\begin{align*}
&  \left(\frac{\al}{2}-1\right)C(\al) \int \frac{\sin(y)h(x-y)}{\left(4\sin^2\left(\frac{y}{2}\right)\right)^{\al/2}}dy \\
& = 2^{-\al+1}\left(\frac{\al}{2}-1\right)C(\al) \int_{0}^{2\pi} \sum_{k=1}^{\infty}a_{k}(\cos(kx-ky))\cos\left(\frac{y}{2}\right)\left(\sin\left(\frac{y}{2}\right)\right)^{-\al+1} dy \\
\end{align*}

Using that

\begin{align*}
\cos\left(\frac{y}{2}\right) \left(\sin\left(\frac{y}{2}\right)\right)^{1-\al} = \frac{2}{2-\al} \pa_y\left(\sin\left(\frac{y}{2}\right)^{2-\al}\right)
\end{align*}

we can integrate by parts to obtain

\begin{align*}
& 2^{-\al+1}\left(\frac{\al}{2}-1\right)C(\al) \int_{0}^{2\pi} \sum_{k=1}^{\infty}\frac{-2k}{2-\al} a_{k}(\sin(kx-ky))\left(\sin\left(\frac{y}{2}\right)\right)^{2-\al} dy \\
& = 2^{-\al+1}\left(\frac{\al}{2}-1\right)C(\al) \sum_{k=1}^{\infty}\frac{-2k}{2-\al} a_{k}\sin(kx)\int_{0}^{2\pi} \cos(ky)\left(\sin\left(\frac{y}{2}\right)\right)^{2-\al} dy \\
& = 2^{-\al+1}\left(\frac{\al}{2}-1\right)C(\al) \sum_{k=1}^{\infty}\frac{-2k}{2-\al} a_{k}\sin(kx) \frac{2\pi \cos(k \pi) \Gamma(3-\al)}{2^{2-\al}\Gamma\left(2+k-\frac{\al}{2}\right)\Gamma\left(2-k-\frac{\al}{2}\right)}, \\
\end{align*}

where we have used the following identity (see \cite{Magnus-Oberhettinger:formeln-satze-speziellen-funktionen}):

\begin{align}
\label{integralgammas}
\int_{0}^{\pi} (\sin(\eta))^{x}e^{iy\eta}d\eta = \frac{\pi e^{i\frac{\pi y}{2}}\Gamma(x+1)}{2^{x}\Gamma\left(1+\frac{x+y}{2}\right)\Gamma\left(1+\frac{x-y}{2}\right)}, \quad \forall x >-1, \quad \forall y \in \mathbb{R}.
\end{align}

Extracting the $k$-th (sine) mode contribution, we obtain:

\begin{align*}
 2^{-\al+1}\left(\frac{\al}{2}-1\right)C(\al) \frac{-2k}{2-\al} a_{k}\frac{2 \pi \cos(k \pi) \Gamma(3-\al)}{2^{2-\al}\Gamma\left(2+k-\frac{\al}{2}\right)\Gamma\left(2-k-\frac{\al}{2}\right)}.
\end{align*}

In the particular case $\al = 1$, by using the identity

\begin{align*}
\Gamma(z)\Gamma(1-z) = \frac{\pi}{\sin(\pi z)},
\end{align*}

the contribution amounts to:

\begin{align*}
 a_k \frac{k}{2} \frac{1}{2\pi} \frac{2 \pi (-1)^{k}}{\Gamma\left(\frac32+k\right)\Gamma\left(\frac32-k\right)}
=  -\frac{a_k}{\pi} \frac{2k}{4k^2 - 1}.
\end{align*}

We move now to the second term. We have

\begin{align}
\label{Omegasecondterm}
& C(\al) \int \frac{\cos(y)(h'(x)-h'(x-y))}{\left(4\sin^2\left(\frac{y}{2}\right)\right)^{\al/2}}dy \nonumber \\
& = 2^{-\al} C(\al) \int \sum_{k=1}^{\infty} (-k a_k )\frac{\cos(y)(\sin(kx) - \sin(kx-ky))}{\left(\sin^2\left(\frac{y}{2}\right)\right)^{\al/2}}dy \nonumber \\
& = 2^{-\al} C(\al) \int \sum_{k=1}^{\infty} (-k a_k )(\sin(kx) - \sin(kx-ky))\left(\sin\left(\frac{y}{2}\right)\right)^{-\al}dy \nonumber \\
& + 2^{-\al+1} C(\al) \int \sum_{k=1}^{\infty} k a_k(\sin(kx) - \sin(kx-ky))\left(\sin\left(\frac{y}{2}\right)\right)^{2-\al}dy \nonumber \\
\end{align}

In order to compute the last two integrals we will use the following lemma:

\begin{lemma}
\label{LemmaICIS}
Let $\al \in (0,2), k \in \mathbb{N}$ and let $IS_{k}(x)$ and $IC_{k}(x)$ be defined as

\begin{align*}
IS_k(x) = \int_{0}^{2\pi} \frac{\sin(kx)-\sin(kx-ky)}{\sin\left(\frac{y}{2}\right)^{\al}} dy, \quad
IC_k(x) = \int_{0}^{2\pi} \frac{\cos(kx)-\cos(kx-ky)}{\sin\left(\frac{y}{2}\right)^{\al}} dy \\
\end{align*}
Then, if $\al \neq 1$:

\begin{align*}
IS_{k}  = \sin(kx)2^{\al} \frac{2\pi \Gamma\left(1-\al\right)}{\Gamma\left(\frac{\al}{2}\right)\Gamma\left(1-\frac{\al}{2}\right)}\left(\frac{\Gamma\left(\frac{\al}{2}\right)}{\Gamma\left(1-\frac{\al}{2}\right)} - \frac{\Gamma\left(k+\frac{\al}{2}\right)}{\Gamma\left(1+k-\frac{\al}{2}\right)}\right), \\
IC_{k}  = \cos(kx)2^{\al} \frac{2\pi \Gamma\left(1-\al\right)}{\Gamma\left(\frac{\al}{2}\right)\Gamma\left(1-\frac{\al}{2}\right)}\left(\frac{\Gamma\left(\frac{\al}{2}\right)}{\Gamma\left(1-\frac{\al}{2}\right)} - \frac{\Gamma\left(k+\frac{\al}{2}\right)}{\Gamma\left(1+k-\frac{\al}{2}\right)}\right),
\end{align*}
and if $\al = 1$:

\begin{align*}
 IS_{k}  = \sin(kx) \sum_{m=1}^{k} \frac{8}{2m-1}, \quad
IC_{k}  = \cos(kx) \sum_{m=1}^{k} \frac{8}{2m-1}.
\end{align*}

\end{lemma}

\begin{proof}
The proof is done by induction. We will find a recurrence for $IS_{k}, IC_{k}$ in terms of $IS_{k-1}, IC_{k-1}$ and then apply the induction hypothesis.

We start with $IS_{k}(x)$. Using the addition formulas for the sine and cosine:

\begin{align*}
IS_{k}(x) & = \cos(x) IS_{k-1}(x) + \sin(x)IC_{k-1}(x) \\
& + \int_{0}^{2\pi} \frac{\sin((k-1)x-(k-1)y)(\cos(x) - \cos(x-y))}{\sin\left(\frac{y}{2}\right)^{\al}} dy \\
& + \int_{0}^{2\pi} \frac{\cos((k-1)x-(k-1)y)(\sin(x) - \sin(x-y))}{\sin\left(\frac{y}{2}\right)^{\al}} dy \\
& = \cos(x) IS_{k-1}(x) + \sin(x)IC_{k-1}(x) + J_{1}(x) + J_{2}(x)
\end{align*}
\begin{align*}
J_{1}(x) & = 2\cos(x)\int_{0}^{2\pi} \sin((k-1)x-(k-1)y)\sin\left(\frac{y}{2}\right)^{2-\al} dy \\
& -2\sin(x)\int_{0}^{2\pi} \sin((k-1)x-(k-1)y)\cos\left(\frac{y}{2}\right)\sin\left(\frac{y}{2}\right)^{1-\al} dy\\
 & = 2\cos(x)\int_{0}^{2\pi} \sin((k-1)x)\cos((k-1)y)\sin\left(\frac{y}{2}\right)^{2-\al} dy \\
& +2\sin(x)\int_{0}^{2\pi} \cos((k-1)x)\sin((k-1)y)\cos\left(\frac{y}{2}\right)\sin\left(\frac{y}{2}\right)^{1-\al} dy \\
& = J_{1,1}(x) + J_{1,2}(x)
\end{align*}
Writing the cosine as a sum of exponentials and applying formula \eqref{integralgammas}, we get
\begin{align*}
J_{1,1}(x) & = 4 (-1)^{k-1} \frac{\pi}{2^{2-\al}} \frac{\Gamma(3-\al)}{\Gamma\left(1+k-\frac{\al}{2}\right)\Gamma\left(3-k-\frac{\al}{2}\right)}\sin((k-1)x)\cos(x)
\end{align*}

Concerning $J_{1,2}$, we first integrate by parts:
\begin{align*}
J_{1,2} & = -\frac{4(k-1)}{2-\al}\int_{0}^{2\pi} \sin(x)\cos((k-1)x)\cos((k-1)y)\sin\left(\frac{y}{2}\right)^{2-\al} dy \\
& = -\frac{8(k-1)}{2-\al}\frac{(-1)^{k-1}\pi}{2^{2-\al}} \frac{\Gamma(3-\al)}{\Gamma\left(1+k-\frac{\al}{2}\right)\Gamma\left(3-k-\frac{\al}{2}\right)}\cos((k-1)x)\sin(x)
\end{align*}

We move on to $J_2$:

\begin{align*}
J_2(x) & = 2\sin(x)\int_{0}^{2\pi} \cos((k-1)x-(k-1)y)\sin\left(\frac{y}{2}\right)^{2-\al} dy\\
& + 2\cos(x)\int_{0}^{2\pi} \cos((k-1)x-(k-1)y)\cos\left(\frac{y}{2}\right)^{1-\al}\sin\left(\frac{y}{2}\right)^{1-\al} dy \\
& = 2\sin(x)\cos((k-1)x)\int_{0}^{2\pi} \cos((k-1)y)\sin\left(\frac{y}{2}\right)^{2-\al} dy\\
& + 2\cos(x)\sin((k-1)x)\int_{0}^{2\pi} \sin((k-1)y)\cos\left(\frac{y}{2}\right)^{1-\al}\sin\left(\frac{y}{2}\right)^{1-\al} dy \\
& = 4 (-1)^{k-1} \frac{\pi}{2^{2-\al}} \frac{\Gamma(3-\al)}{\Gamma\left(1+k-\frac{\al}{2}\right)\Gamma\left(3-k-\frac{\al}{2}\right)}\cos((k-1)x)\sin(x)\\
& -\frac{8(k-1)}{2-\al}\frac{(-1)^{k-1}\pi}{2^{2-\al}} \frac{\Gamma(3-\al)}{\Gamma\left(1+k-\frac{\al}{2}\right)\Gamma\left(3-k-\frac{\al}{2}\right)}\sin((k-1)x)\cos(x)
\end{align*}

Adding up $J_1(x) + J_2(x)$:

\begin{align*}
J_1(x) + J_2(x) & = \sin(kx)\frac{(-1)^{k-1}\pi}{2^{2-\al}} \frac{\Gamma(3-\al)}{\Gamma\left(1+k-\frac{\al}{2}\right)\Gamma\left(3-k-\frac{\al}{2}\right)}\left(4  -\frac{8(k-1)}{2-\al} \right).
\end{align*}

We calculate the recurrence relation for $IC_k(x)$. Using the same expansion as before, we obtain

\begin{align*}
IC_{k}(x) & = \cos(x) IC_{k-1}(x) - \sin(x)IS_{k-1}(x) \\
& + \int_{0}^{2\pi} \frac{\cos((k-1)x-(k-1)y)(\cos(x) - \cos(x-y))}{\sin\left(\frac{y}{2}\right)^{\al}} dy \\
& - \int_{0}^{2\pi} \frac{\sin((k-1)x-(k-1)y)(\sin(x) - \sin(x-y))}{\sin\left(\frac{y}{2}\right)^{\al}} dy \\ & = \cos(x) IC_{k-1}(x) - \sin(x)IS_{k-1}(x) + K_1(x) + K_2(x)
\end{align*}

We continue with $K_1(x)$:

\begin{align*}
K_1(x) & = 2\cos(x)\int_{0}^{2\pi} \cos((k-1)x-(k-1)y)\sin\left(\frac{y}{2}\right)^{2-\al} dy\\
& - 2\sin(x)\int_{0}^{2\pi} \cos((k-1)x-(k-1)y)\cos\left(\frac{y}{2}\right)^{1-\al}\sin\left(\frac{y}{2}\right)^{1-\al} dy \\
& = 2\cos(x)\cos((k-1)x)\int_{0}^{2\pi} \cos((k-1)y)\sin\left(\frac{y}{2}\right)^{2-\al} dy\\
& - 2\sin(x)\sin((k-1)x)\int_{0}^{2\pi} \sin((k-1)y)\cos\left(\frac{y}{2}\right)^{1-\al}\sin\left(\frac{y}{2}\right)^{1-\al} dy \\
& = 4 (-1)^{k-1} \frac{\pi}{2^{2-\al}} \frac{\Gamma(3-\al)}{\Gamma\left(1+k-\frac{\al}{2}\right)\Gamma\left(3-k-\frac{\al}{2}\right)}\cos((k-1)x)\cos(x)\\
& +\frac{8(k-1)}{2-\al}\frac{(-1)^{k-1}\pi}{2^{2-\al}} \frac{\Gamma(3-\al)}{\Gamma\left(1+k-\frac{\al}{2}\right)\Gamma\left(3-k-\frac{\al}{2}\right)}\sin((k-1)x)\sin(x)
\end{align*}

In a similar way, $K_2(x)$ is equal to:

\begin{align*}
K_2(x) & = -2\sin(x)\int_{0}^{2\pi} \sin((k-1)x-(k-1)y)\sin\left(\frac{y}{2}\right)^{2-\al} dy\\
& - 2\cos(x)\int_{0}^{2\pi} \sin((k-1)x-(k-1)y)\cos\left(\frac{y}{2}\right)^{1-\al}\sin\left(\frac{y}{2}\right)^{1-\al} dy \\
& = -2\sin(x)\sin((k-1)x)\int_{0}^{2\pi} \cos((k-1)y)\sin\left(\frac{y}{2}\right)^{2-\al} dy\\
& + 2\cos(x)\cos((k-1)x)\int_{0}^{2\pi} \sin((k-1)y)\cos\left(\frac{y}{2}\right)^{1-\al}\sin\left(\frac{y}{2}\right)^{1-\al} dy \\
& = -4 (-1)^{k-1} \frac{\pi}{2^{2-\al}} \frac{\Gamma(3-\al)}{\Gamma\left(1+k-\frac{\al}{2}\right)\Gamma\left(3-k-\frac{\al}{2}\right)}\sin((k-1)x)\sin(x)\\
& -\frac{8(k-1)}{2-\al}\frac{(-1)^{k-1}\pi}{2^{2-\al}} \frac{\Gamma(3-\al)}{\Gamma\left(1+k-\frac{\al}{2}\right)\Gamma\left(3-k-\frac{\al}{2}\right)}\cos((k-1)x)\cos(x)
\end{align*}

Adding up $K_1(x)$ and $K_2(x)$:

\begin{align*}
K_1(x) + K_2(x)  & =
4 (-1)^{k-1} \frac{\pi}{2^{2-\al}} \frac{\Gamma(3-\al)}{\Gamma\left(1+k-\frac{\al}{2}\right)\Gamma\left(3-k-\frac{\al}{2}\right)}\cos(kx)\\
& -\frac{8(k-1)}{2-\al}\frac{(-1)^{k-1}\pi}{2^{2-\al}} \frac{\Gamma(3-\al)}{\Gamma\left(1+k-\frac{\al}{2}\right)\Gamma\left(3-k-\frac{\al}{2}\right)}\cos(kx) \\
& = \cos(kx)\frac{(-1)^{k-1}\pi}{2^{2-\al}} \frac{\Gamma(3-\al)}{\Gamma\left(1+k-\frac{\al}{2}\right)\Gamma\left(3-k-\frac{\al}{2}\right)}\left(4 -\frac{8(k-1)}{2-\al} \right).
\end{align*}

We distinguish two cases. In the case $\al = 1$, we can simplify our formulas by using that

\begin{align*}
 \left.\frac{(-1)^{k-1}\pi}{2^{2-\al}} \frac{\Gamma(3-\al)}{\Gamma\left(1+k-\frac{\al}{2}\right)\Gamma\left(3-k-\frac{\al}{2}\right)}\left(4 -\frac{8(k-1)}{2-\al} \right)\right|_{\al=1}
= \left.(-1)^{k-1}\pi 2^{1+\al} \frac{\Gamma(2-\al)}{\Gamma\left(1+k-\frac{\al}{2}\right)\Gamma\left(2-k-\frac{\al}{2}\right)}\right|_{\al=1} \\
= (-1)^{k-1}\frac{4\pi}{\Gamma\left(\frac{1}{2}+k\right)\Gamma\left(\frac{3}{2}-k\right)}
= (-1)^{k-1}\frac{4\pi}{\Gamma\left(\frac{1}{2}+k\right)\Gamma\left(\frac{1}{2}-k\right)\left(\frac{1}{2}-k\right)}
= \frac{8}{2k-1},
\end{align*}

where in the last equality we have used the identity
\begin{align*}
 \Gamma(1-z)\Gamma(z) = \frac{\pi}{\sin(\pi z)}.
\end{align*}

Adding in $k$, we obtain the desired formulas for $IS, IC$:

\begin{align*}
 IS_{k}  = \sin(kx) \sum_{m=1}^{k} \frac{8}{2m-1}, \quad
IC_{k}  = \cos(kx) \sum_{m=1}^{k} \frac{8}{2m-1}.
\end{align*}

For the other values of $\alpha$, we use induction. We start by checking the base case ($k = 1$):

\begin{align*}
\frac{2^{\al} \pi \Gamma(3-\al)}{\left(\Gamma\left(2-\frac{\al}{2}\right)\right)^2}
& = \frac{2^{\al} 2\pi \Gamma(1-\al)}{\Gamma\left(1-\frac{\al}{2}\right)}\left(\frac{1-\al}{\left(1-\frac{\al}{2}\right)\Gamma\left(1-\frac{\al}{2}\right)}\right)
 = \frac{2^{\al} 2\pi \Gamma(1-\al)}{\Gamma\left(1-\frac{\al}{2}\right)}\left(\frac{1}{\Gamma\left(1-\frac{\al}{2}\right)}\left(\frac{\Gamma\left(\frac{\al}{2}\right)}{\Gamma\left(\frac{\al}{2}\right)}-\frac{\Gamma\left(1+\frac{\al}{2}\right)}{\left(1-\frac{\al}{2}\right)\Gamma\left(\frac{\al}{2}\right)}\right)\right) \\
 &  = \frac{2^{\al} 2\pi \Gamma(1-\al)}{\Gamma\left(1-\frac{\al}{2}\right)\Gamma\left(\frac{\al}{2}\right)}\left(\frac{\Gamma\left(\frac{\al}{2}\right)}{\Gamma\left(1-\frac{\al}{2}\right)}-\frac{\Gamma\left(1+\frac{\al}{2}\right)}{\Gamma\left(2-\frac{\al}{2}\right)}\right)
\end{align*}

Finally, we do the induction step. We assume that the formula is true for $k-1$ ($k\geq 2$) and we show it for $k$. It is enough to check that:

\begin{align*}
& & \frac{2^{\al} 2\pi \Gamma(1-\al)}{\Gamma\left(1-\frac{\al}{2}\right)\Gamma\left(\frac{\al}{2}\right)}\left(\frac{\Gamma\left(\frac{\al}{2}\right)}{\Gamma\left(1-\frac{\al}{2}\right)}-\frac{\Gamma\left(k-1+\frac{\al}{2}\right)}{\Gamma\left(k-\frac{\al}{2}\right)}\right) & \\
& &  + \frac{2^{\al} \pi \Gamma(3-\al)(-1)^{k-1} \left(1 - \frac{2(k-1)}{2-\al}\right)}{\Gamma\left(k+1-\frac{\al}{2}\right)\Gamma\left(3-k-\frac{\al}{2}\right)}&  = \frac{2^{\al} 2\pi \Gamma(1-\al)}{\Gamma\left(1-\frac{\al}{2}\right)\Gamma\left(\frac{\al}{2}\right)}\left(\frac{\Gamma\left(\frac{\al}{2}\right)}{\Gamma\left(1-\frac{\al}{2}\right)}-\frac{\Gamma\left(k+\frac{\al}{2}\right)}{\Gamma\left(k+1-\frac{\al}{2}\right)}\right) \\
\Leftrightarrow & & \frac{2^{\al} 2\pi \Gamma(1-\al)}{\Gamma\left(1-\frac{\al}{2}\right)\Gamma\left(\frac{\al}{2}\right)}\left(\frac{\Gamma\left(k+\frac{\al}{2}\right)}{\Gamma\left(k+1-\frac{\al}{2}\right)}-\frac{\Gamma\left(k-1+\frac{\al}{2}\right)}{\Gamma\left(k-\frac{\al}{2}\right)}\right)  & = \frac{2^{\al} \pi (2-\al)(1-\al)\Gamma(1-\al)(-1)^{k} \left(1 - \frac{2(k-1)}{2-\al}\right)}{\Gamma\left(k+1-\frac{\al}{2}\right)\Gamma\left(3-k-\frac{\al}{2}\right)} \\
\Leftrightarrow & & \frac{ 2}{\Gamma\left(1-\frac{\al}{2}\right)\Gamma\left(\frac{\al}{2}\right)}\frac{\Gamma\left(k+\frac{\al}{2}\right)}{\Gamma\left(k+1-\frac{\al}{2}\right)}\left(1-\frac{k-\frac{\al}{2}}{k-1+\frac{\al}{2}}\right)  & = \frac{(2-\al)(1-\al)(-1)^{k} \left(1 - \frac{2(k-1)}{2-\al}\right)}{\Gamma\left(k+1-\frac{\al}{2}\right)\Gamma\left(3-k-\frac{\al}{2}\right)} \\
\Leftrightarrow & & \frac{2\Gamma\left(k-1+\frac{\al}{2}\right)}{\Gamma\left(1-\frac{\al}{2}\right)\Gamma\left(\frac{\al}{2}\right)}  & = \frac{(-1)^{k+1}(4-\al-2k)}{\Gamma\left(3-k-\frac{\al}{2}\right)} \\
\Leftrightarrow & & \frac{\Gamma\left(k-1+\frac{\al}{2}\right)}{\Gamma\left(\frac{\al}{2}\right)}  & = (-1)^{k+1}\left(2-k-\frac{\al}{2}\right)\frac{\Gamma\left(1-\frac{\al}{2}\right)}{\Gamma\left(3-k-\frac{\al}{2}\right)} \\
\Leftrightarrow & & \left(k-2+\frac{\al}{2}\right)\left(k-3+\frac{\al}{2}\right) \cdots \left(\frac{\al}{2}\right)  & = (-1)^{k+1}\left(2-k-\frac{\al}{2}\right)\left(-\frac{\al}{2}\right)\left(-\frac{\al}{2}-1\right) \cdots \left(-\frac{\al}{2}-(k-3)\right),
\end{align*}

which is true. This finishes the proof.

\end{proof}

We insert the previous result in \eqref{Omegasecondterm} and extract the $k$-th mode contribution. In the case $\al = 1$:

\begin{align*}
& a_k \frac{1}{4\pi} (-k) \sum_{m=1}^{k}\frac{8}{2m-1}
+ k a_k \frac{1}{2\pi} \frac{1}{2} \frac{2\pi \Gamma(2)}{\Gamma\left(-\frac12\right)\Gamma\left(\frac32\right)}\left(\frac{\Gamma\left(-\frac12\right)}{\Gamma\left(\frac32\right)} - \frac{\Gamma\left(k-\frac12\right)}{\Gamma\left(k+\frac32\right)}\right) \\
& = -k a_k \frac{2}{\pi}\sum_{m=1}^{k}\frac{1}{2m-1}
+ k a_k  \frac{1}{2\pi}\left(4 + \frac{1}{k^2 - \frac{1}{4}}\right). \\
\end{align*}

Combining the sum of every contribution, they amount to

\begin{align*}
-k a_k  \Omega
-\frac{a_k}{\pi} \frac{2k}{4k^2 - 1}
-k a_k \frac{2}{\pi}\sum_{m=1}^{k}\frac{1}{2m-1}
+ k a_k  \frac{1}{2\pi}\left(4 + \frac{1}{k^2 - \frac{1}{4}}\right)
= -ka_k [\Omega - \Omega_k]. \\
\end{align*}

This proves the case $\al = 1$. In the case $\al \neq 1$, the coefficient in front of the $k$-th mode is:

\begin{align*}
& -k a_k  \Omega
+ 2^{-\al+1}\left(\frac{\al}{2}-1\right)C(\al) \frac{-2k}{2-\al} a_{k}\frac{2 \pi \cos(k \pi) \Gamma(3-\al)}{2^{2-\al}\Gamma\left(2+k-\frac{\al}{2}\right)\Gamma\left(2-k-\frac{\al}{2}\right)} \\
 & a_k 2^{-\al} C(\al) (-k) 2^{\al} \frac{2\pi \Gamma(1-\al)}{\Gamma\left(\frac{\al}{2}\right)\Gamma\left(1-\frac{\al}{2}\right)}\left(\frac{\Gamma\left(\frac{\al}{2}\right)}{\Gamma\left(1-\frac{\al}{2}\right)} - \frac{\Gamma\left(k+\frac{\al}{2}\right)}{\Gamma\left(1+k-\frac{\al}{2}\right)}\right)\\
& + a_k 2^{-\al+1} C(\al) k 2^{\al-2} \frac{2\pi \Gamma(3-\al)}{\Gamma\left(\frac{\al}{2}-1\right)\Gamma\left(2-\frac{\al}{2}\right)}\left(\frac{\Gamma\left(\frac{\al}{2}-1\right)}{\Gamma\left(2-\frac{\al}{2}\right)} - \frac{\Gamma\left(k-1+\frac{\al}{2}\right)}{\Gamma\left(2+k-\frac{\al}{2}\right)}\right).
\end{align*}

We can group the third and the fifth factor into

\begin{align*}
 & a_k 2^{-\al} C(\al) (-k) 2^{\al} \frac{2\pi \Gamma(1-\al)}{\Gamma\left(\frac{\al}{2}\right)\Gamma\left(1-\frac{\al}{2}\right)}\left(\frac{\Gamma\left(\frac{\al}{2}\right)}{\Gamma\left(1-\frac{\al}{2}\right)}\right)
 + a_k 2^{-\al+1} C(\al) k 2^{\al-2} \frac{2\pi \Gamma(3-\al)}{\Gamma\left(\frac{\al}{2}-1\right)\Gamma\left(2-\frac{\al}{2}\right)}\left(\frac{\Gamma\left(\frac{\al}{2}-1\right)}{\Gamma\left(2-\frac{\al}{2}\right)}\right) \\
= & k a_k 2^{\al - 1} \frac{\Gamma(1-\al)}{\left(\Gamma\left(1-\frac{\al}{2}\right)\right)^{2}} \left(-\frac{\Gamma\left(\frac{\al}{2}\right)}{\Gamma\left(1-\frac{\al}{2}\right)} + \frac12 \frac{(2-\al)(1-\al) \Gamma\left(\frac{\al}{2}\right)}{\left(1-\frac{\al}{2}\right)\Gamma\left(2-\frac{\al}{2}\right)}\right) \\
= & k a_k 2^{\al - 1} \frac{\Gamma(1-\al)}{\left(\Gamma\left(1-\frac{\al}{2}\right)\right)^{2}} \frac{\Gamma\left(1+\frac{\al}{2}\right)}{\Gamma\left(2-\frac{\al}{2}\right)} \left(-\frac{1-\frac{\al}{2}}{\frac{\al}{2}} + \frac{1}{2}\frac{(2-\al)(1-\al)}{\left(1-\frac{\al}{2}\right)\left(\frac{\al}{2}\right)}\right) \\
= & -k a_k 2^{\al - 1} \frac{\Gamma(1-\al)}{\left(\Gamma\left(1-\frac{\al}{2}\right)\right)^{2}} \frac{\Gamma\left(1+\frac{\al}{2}\right)}{\Gamma\left(2-\frac{\al}{2}\right)}, \\
\end{align*}

and the second, fourth and sixth as

\begin{align*}
& k a_k 2^{\al-1} \frac{\Gamma(1-\al)}{\left(\Gamma\left(1-\frac{\al}{2}\right)\right)^{2}} \frac{\Gamma\left(k+\frac{\al}{2}\right)}{\Gamma\left(1+k-\frac{\al}{2}\right)} \\
& \times \left(1 - \frac12 \frac{(2-\al)(1-\al)}{\left(1-\frac{\al}{2}\right)\left(k-1+\frac{\al}{2}\right)\left(1+k-\frac{\al}{2}\right)} \frac{\Gamma\left(\frac{\al}{2}\right)}{\Gamma\left(\frac{\al}{2}-1\right)}+ \frac12 \frac{(2-\al)(1-\al)\Gamma\left(1-\frac{\al}{2}\right)\Gamma\left(\frac{\al}{2}\right)(-1)^{k}}{\left(1+k-\frac{\al}{2}\right)\Gamma\left(k+\frac{\al}{2}\right)\Gamma\left(2-k-\frac{\al}{2}\right)}\right) \\
& =  k a_k 2^{\al-1} \frac{\Gamma(1-\al)}{\left(\Gamma\left(1-\frac{\al}{2}\right)\right)^{2}} \frac{\Gamma\left(k+\frac{\al}{2}\right)}{\Gamma\left(1+k-\frac{\al}{2}\right)}
\left(1 - \frac12 \frac{(1-\al)(\al-2)}{\left(k-1+\frac{\al}{2}\right)\left(k+1-\frac{\al}{2}\right)} + \frac12 \frac{(1-\al)(\al-2)}{\left(k-1+\frac{\al}{2}\right)\left(k+1-\frac{\al}{2}\right)}\right) \\
& =  k a_k 2^{\al-1} \frac{\Gamma(1-\al)}{\left(\Gamma\left(1-\frac{\al}{2}\right)\right)^{2}} \frac{\Gamma\left(k+\frac{\al}{2}\right)}{\Gamma\left(1+k-\frac{\al}{2}\right)}
\end{align*}

In total, we get that the $k$-th coefficient is precisely

\begin{align*}
 -k a_k(\Omega - \Omega_k),
\end{align*}

as claimed.
\end{proof}

\begin{prop}
\label{Omegammonotonic}
Let $\al \in (0,2)$. The values of $\Omega_k$ are monotonic with $k$.
%
\end{prop}
\begin{proof} The case $\alpha = 1$ is trivial and was already covered in \cite{Hassainia-Hmidi:v-states-generalized-sqg}. For the rest of the values of $\alpha$, it is enough to show that $\frac{\Gamma\left(k+\frac{\al}{2}\right)}{\Gamma\left(1+k-\frac{\al}{2}\right)}$ is monotonic with $k$. But we have

\begin{align*}
\frac{\Gamma\left(1+k+\frac{\al}{2}\right)}{\Gamma\left(2+k-\frac{\al}{2}\right)}
 = \frac{\Gamma\left(k+\frac{\al}{2}\right)}{\Gamma\left(1+k-\frac{\al}{2}\right)} \frac{k+\frac{\al}{2}}{1+k-\frac{\al}{2}}
 = \frac{\Gamma\left(k+\frac{\al}{2}\right)}{\Gamma\left(1+k-\frac{\al}{2}\right)} \left(1 + \frac{\al-1}{1+k-\frac{\al}{2}}\right).
\end{align*}

In the case $\alpha > 1$, the bracket is strictly greater than 1, and in the case $\alpha < 1$, the bracket is strictly smaller than 1, independently of $k$. Monotonicity for all cases follows.

\end{proof}

From Proposition \ref{representationOperator} and Proposition \ref{Omegammonotonic} it is immediate that if $\Omega = \Omega_{m}$, then the kernel is non trivial, has dimension 1, and it is generated by $\cos(mx)$.

We continue by computing the range of $D_{R} F(\Omega_m,1)$. To do so, we will prove that the range is indeed the set

\begin{align*}
\left\{
\begin{array}{cc}
\displaystyle  Z_{m} = \left\{f \in Y_{m}^{k-1}, f = \sum_{k>1}^{\infty}a_k \sin(kmx)\right\}, \text{ if } \al = 1 \\
 \displaystyle Z_{m} = \left\{f \in Y_{m}^{k-\al}, f = \sum_{k>1}^{\infty}a_k \sin(kmx)\right\}, \text{ if } \al > 1 \\
\end{array}
\right.
\end{align*}

If we are able to do so, then we are done since $Z_{m}$ is closed and has codimension 1 in $Y^{k-1}_{m}$ or $Y^{k-\al}_{m}$, depending on whether we are in the case $\al = 1$ or $\al > 1$. We note that by Proposition \ref{representationOperator}, the inclusion Range$(D_{R}F(\Omega_m,1)) \subset Z_{m} $ follows trivially. We now show the opposite one.

Let $g \in Z_{m}, g = \sum_{k>1}^{\infty}g_k\sin(kmx)$. We need to show that there exists an $h$ such that $D_{R}F(\Omega_m,1)(h) = g$. However, by the representation given by Proposition \ref{representationOperator}, such an $h$ exists and it is given by

\begin{align*}
 h(x) = \sum_{k>1}^{\infty}h_k\cos(kmx), \quad h_k = \frac{g_k}{km(\Omega_{km} - \Omega_m)}.
\end{align*}

We have to check that $h$ has the right regularity: for the case $\al = 1$, this will mean that $h \in X^{k+\log}$, whereas for the case $\al > 1$, $h$ will have to belong to $H^{k}$. In order to establish that condition, the following Lemma will be useful.

\begin{lemma}
\label{lemmagrowthomega}
In the case $\al = 1$, $\Omega_{m} \sim \log(m)$, and for $\al > 1$, $\Omega_{m} \sim m^{\al-1}$
\end{lemma}
\begin{proof}
The case $\al = 1$ was proved in \cite{Hassainia-Hmidi:v-states-generalized-sqg} and the case $\al > 1$ follows directly from the asymptotic expansion of the Gamma function given by \cite[Formula 6.1.46,p.257]{Abramowitz-Stegun:handbook-mathematical-functions}.
\end{proof}

In the case $\al = 1$, we use the alternative characterization of the $X^{k+\log}$ spaces in Proposition \ref{alternativexlog}, and the asymptotic growth of $\Omega_m$ from Lemma
\ref{lemmagrowthomega} and bound the following quantity:

\begin{align*}
 \sum_{p>1}|h_p|^{2}(pm)^{2k}(1+\log(pm))^{2}
& = \sum_{p>1}\left|\frac{g_p}{pm(\Omega_{pm} - \Omega_m)}\right|^{2}(pm)^{2k}(1+\log(pm))^{2} \\
& \leq C \sum_{p>1} |g_p|^{2} (pm)^{2k-2} \left(\frac{1+\log(pm)}{\log(pm)}\right)^{2} \\
& \leq C \sum_{p>1} |g_p|^{2} (pm)^{2k-2} = C\|g\|_{H^{k-1}} < \infty
\end{align*}

In the case $\al > 1$, we compute the $H^{k}$ norm squared of $h$ and obtain

\begin{align*}
 \sum_{p>1}|h_p|^{2}(pm)^{2k} = \sum_{p>1}\left|\frac{g_p}{pm(\Omega_{pm} - \Omega_m)}\right|^{2}(pm)^{2k}
\leq C \sum_{p>1}\left|\frac{h_p}{pm(pm)^{\al-1}}\right|^{2}(pm)^{2k} \leq C \|g\|_{H^{k-\al}}^{2} < \infty
\end{align*}

This shows step 4.

\subsection{Step 5}

We will show step 5 using the previous characterization. First of all, we recall that

\begin{align*}
 F_{\Omega R}(\Omega_m,1) (h) = h'(x).
\end{align*}

Therefore,

\begin{align*}
 F_{\Omega R}(\Omega_m,1) (\cos(mx)) = -m\sin(mx),
\end{align*}

which does not belong to Range($\mathcal{F}$), as we wanted to prove.

\subsection{Step 6}

Since in Step 1 we showed the regularity, the only thing that is left is to show that $F(\Omega,R)$ has $m$-fold symmetry and can be written as a Fourier-sin series. To do so, we will use the following lemmas:

\begin{lemma}
\label{lemmaoddnessalpha1}
Let $\alpha \geq 1$. If $R(x)$ is even, then $F(\Omega,R)$ is odd.
\end{lemma}
\begin{proof}
The first term $\Omega R'(x)$ is clearly odd. To see the oddness of the other ones, it is enough to compute $F_i(R)(x)$ and $-F_i(R)(-x)$. One is obtained from the other by changing $y \mapsto -y$ and using the fact that $R(x)$ is even and $R'(x)$ is odd.

\end{proof}
\begin{corollary}
The Fourier series of $F(\Omega,R)$ consists only of sine terms.
\end{corollary}

\begin{lemma}
\label{lemmasymmetryalpha1}
Let $\alpha \geq 1$. If $R(x)$ is expressed as an $m$-fold series of cosines, then $F(\Omega,R)(x+\frac{2\pi}{m}) = F(\Omega,R)(x)$.
\end{lemma}
\begin{proof}

The first term, $\Omega R'(x)$ satisfies

\begin{align*}
\Omega R'\left(x+\frac{2\pi}{m}\right) = \Omega R'(x),
\end{align*}

To check the property of the other terms, it is enough to compute $F_i(R)(x)$ and $F_i(R)\left(x+\frac{2\pi}{m}\right)$. One is obtained from the other by changing $y \mapsto y+\frac{2\pi}{m}$ and using the fact that $R(x) = R\left(x+\frac{2\pi}{m}\right)$ and $R'(x) = R'\left(x+\frac{2\pi}{m}\right)$.

%
%

\end{proof}

\end{proof}

\section{Existence for $1<\alpha <2$.}
\label{Sectionexistencealphamayor1}
This section is devoted to show Theorem \ref{teoremaexistenciaalphamayor1}.

\begin{proof}
The proof of this theorem follows the same steps that the proof of Theorem \ref{teoremaexistenciaalpha1}. It will be divided into 6 steps. These steps correspond to check the hypotheses of the Crandall-Rabinowitz theorem \cite{Crandall-Rabinowitz:bifurcation-simple-eigenvalues} for $$F(\Omega,R)=\Omega R'-\sum_{i=1}^3F_i(R),$$
where
\begin{align*}
F_1(R)=&\frac{C(\alpha)}{R(x)}\int\frac{\sin(x-y)}{\left(\left(R(x)-R(y)\right)^2+4R(x)R(y)\sin^2\left(\frac{x-y}{2}\right)\right)^\frac{\alpha}{2}}
\left(R(x)R(y)+R'(x)R'(y)\right)dy,\\
F_2(R)=&C(\alpha)\int\frac{\cos(x-y)}{\left(\left(R(x)-R(y)\right)^2+4R(x)R(y)\sin^2\left(\frac{x-y}{2}\right)\right)^\frac{\alpha}{2}}
\left(R'(y)-R'(x)\right)dy,\\
F_3(R)=&C(\alpha)\frac{R'(x)}{ R(x) }\int\frac{\cos(x-y)}{\left(\left(R(x)-R(y)\right)^2+4R(x)R(y)\sin^2\left(\frac{x-y}{2}\right)\right)^\frac{\alpha}{2}}
\left(R(x)-R(y)\right)dy,
\end{align*}
and they are the following
\begin{enumerate}
\item The functional $F$ satisfies $$F(\Omega,R)\,:\, \R\times \{1+V^r\}\mapsto Y^{k-1},$$ where $V^r$ is the open neighborhood of 0
$$V^r=\{ f\in X^{k+\alpha-1}\,:\, ||f||_{H^{k+\alpha-1}}<r\},$$
for  $0<r<1$ and $k\geq 3$.
\item $F(\Omega,1) = 0$ for every $\Omega$.
\item The partial derivatives $F_{\Omega}$, $F_{R}$ and $F_{R\Omega}$ exist and are continuous.
\item Ker($\mathcal{F}$) and $Y^{k-1}$/Range($\mathcal{F}$) are one-dimensional, where $\mathcal{F}$ is the linearized operator around the disk $R = 1$ at $\Omega = \Omega_m$.
\item $F_{\Omega R}(\Omega_m,1)(h_0) \not \in$ Range($\mathcal{F}$), where Ker$(\mathcal{F}) = \langle h_0 \rangle$.
\item Step 1 can be applied to the spaces $X^{k+\al-1}_{m}$ and $Y^{k-1}_{m}$ instead of $X^{k+\al-1}$ and $Y^{k-1}$.
\end{enumerate}

\subsection{Step 1}
\begin{prop}\label{prop2}
Let $0 < r < 1$, $k \geq 3$.  Then
 $$F(\Omega,R): \mathbb{R} \times \{1+V^r\} \mapsto Y^{k-1}.$$
\end{prop}
\begin{proof}
We recall that the norm $||f||_{H^{k+\alpha-1}}$ can be defined as
\begin{align*}
||f||_{H^{k+\alpha-1}}=||f||_{H^k}+\left|\left|\int \frac{\pa^kf(\cdot)-\pa^k f(y)}{|\sin\left(\frac{\cdot-y}{2}\right)|^\alpha} dy\right|\right|_{L^2}.
\end{align*}

\begin{rem} Indeed we saw in Lemma \ref{LemmaICIS} that this definition is equivalent to
$$ ||f||_{H^{k+\alpha-1}}=||f||_{H^k} + || |m|^{\alpha-1}\hat{f}_m||_{l^2}.$$
\end{rem}
Using this definition makes quite similar the proofs of Theorems \ref{teoremaexistenciaalpha1} and \ref{teoremaexistenciaalphamayor1}.

We will use the following decomposition
\begin{align*}
\frac{1}{\left(\left(R(x)-R(y)\right)^2+4R(x)R(y)\sin^2\left(\frac{x-y}{2}\right)\right)^\frac{\alpha}{2}}
= K_{S}(x,y)+\frac{1}{\left(R(x)^2+R'(x)^2\right)^\frac{\alpha}{2}}\frac{1}{2^\alpha\left|\sin\left(\frac{x-y}{2}\right)\right|^\alpha},
\end{align*}
where the kernel
\begin{align*}
K_S(x,y)\equiv \frac{1}{\left(\left(R(x)-R(y)\right)^2+4R(x)R(y)\sin^2\left(\frac{x-y}{2}\right)\right)^\frac{\alpha}{2}}-
\frac{1}{\left(R(x)^2+R'(x)^2\right)^\frac{\alpha}{2}}\frac{1}{2^\alpha\left|\sin\left(\frac{x-y}{2}\right)\right|^\alpha}
\end{align*}

belongs to $\mathcal{H}_0$.

Again the most singular term is $\pa^{k-1}F_2$. Making the change of variable  $x-y \mapsto y$, taking $\pa^{k-1}$ derivatives with respect to $x$  and changing again to $y\mapsto x-y$ yields
\begin{align*}
\pa^{k-1} F_{2}(R)=& C(\alpha)\int\frac{\cos(x-y)}{\left(\left(R(y)-R(x)\right)^2+4R(x)R(y)\sin^2\left(\frac{x-y}{2}\right)\right)^\frac{\alpha}{2}}
\left(\pa^k R(x)-\pa^k R(y)\right)dy\\
&+ \text{l.o.t.}
\end{align*}
We will split the first term as follows
\begin{align*}
\pa^{k-1} F_{2}(R)=&C(\alpha)\int \cos(x-y)K_{S}(x,y)\left(\pa^k R(y)-\pa^kR(x)\right)dy\\
&-C(\alpha)\int \frac{\left|\sin\left(\frac{x-y}{2}\right)\right|}{\left(R(x)^2+R'(x)^2\right)^\frac{\alpha}{2}}\left(\pa^k R(y)-\pa^k R(x)\right)dy\\
&+C(\alpha)\frac{1}{2^\alpha \left(R(x)^2+R'(x)^2\right)^\frac{\al}{2}} \int \frac{(\pa^k R(y)-\pa^k R(x))}{\left|\sin\left(\frac{x-y}{2}\right)\right|^{\alpha}}.
\end{align*}
Therefore,
 because of the definition of the space $X^{k+\alpha-1}$ we have that
\begin{align*}
||\pa^{k-1} F_2(R)||_{L^2(\T)}\leq C\left(||R||_{X^{k+\log}},r\right).
\end{align*}
\end{proof}
\subsection{Step 2}
Again it is trivial to prove that $F(\Omega,1)=0$.
\subsection{Step 3}
We need to prove the existence and the continuity of the Gateaux derivatives $\pa_\Omega F(\Omega,R)$, $\pa_R F(\Omega,R)$ and $\pa_{\Omega,R}F(\Omega,R)$. In order to do it the most difficult part is to show the existence and continuity of $\pa_R F_i(R)$ for $i=1,2,3$.
\begin{lemma} \label{gatoderivadamayor1}
For all $R\in V^r$ and for all $h\in X^{k+\al-1}$ such that $||h||_{X^{k+\al-1}}=1$ we have that
$$\lim_{t \to 0}\frac{F_i(R+th)-F_i(R)}{t}= D_i[R]h\quad \text{in $Y^{k-1}$},$$
where
\begin{align*}
D_1[R] h =& -C(\alpha)\frac{h(x)}{ R(x)^2}\int \frac{\sin(x-y)}{\left((R(x)-R(y))^2+4R(x)R(y)\sin^2\left(\frac{x-y}{2}\right)\right)^\frac{\alpha}{2}} \left(R(x)R(y)+R'(x)R'(y)\right)dy\\
& +C(\alpha)\frac{1}{ R(x)}\int \frac{\sin(x-y)}{\left((R(x)-R(y))^2+4R(x)R(y)\sin^2\left(\frac{x-y}{2}\right)\right)^\frac{\alpha}{2}}\\&\times(h(x)R(y)+h(y)R(x)+(h'(x)R'(y)+h'(y)R'(x)))dy\\
&-\alpha C(\alpha)\frac{1}{ R(x)}\int \frac{\sin(x-y)(R(x)R(y)+R'(x)R'(y))}{\left((R(x)-R(y))^2+4R(x)R(y)\sin^2\left(\frac{x-y}{2}\right)\right)^{\frac{\alpha}{2}+1}}\\
&\times\left((R(x)-R(y))(h(x)-h(y))+2(h(x)R(y)+h(y)R(x))\sin^2\left(\frac{x-y}{2}\right)\right)dy\\
D_2[R] h = & C(\alpha)\int\frac{\cos(x-y)}{\left((R(x)-R(y))^2+4R(x)R(y)\sin^2\left(\frac{x-y}{2}\right)\right)^\frac{\alpha}{2}}\left(h'(y)-h'(x)\right)dy\\
&-\alpha C(\alpha)\int\frac{\cos(x-y)(R'(y)-R'(x))}{\left((R(x)-R(y))^2+
4R(x)R(y)\sin\left(\frac{x-y}{2}\right)\right)^{\frac{\alpha}{2}+1}}\\ & \times \left((R(x)-R(y))(h(x)-h(y))+2(h(x)R(y)+h(y)R(x))\sin^2\left(\frac{x-y}{2}\right)\right)dy\\
D_3[R]h= & C(\alpha)\frac{h'(x)}{ R(x)}\int\frac{\cos(x-y)}{\left(\left(R(x)-R(y)\right)^2+4R(x)R(y)\sin^2\left(\frac{x-y}{2}\right)\right)^\frac{\alpha}{2}}(R(x)-R(y))dy\\
&-C(\alpha)\frac{R'(x) h(x)}{ R(x)^2}\int\frac{\cos(x-y)}{\left(\left(R(x)-R(y)\right)^2+4R(x)R(y)\sin^2\left(\frac{x-y}{2}\right)\right)^\frac{\alpha}{2}}
(R(x)-R(y)) dy \\
&+C(\alpha)\frac{R'(x)}{ R(x)}\int\frac{\cos(x-y)}{\left(\left(R(x)-R(y)\right)^2+4R(x)R(y)\sin^2\left(\frac{x-y}{2}\right)\right)^\frac{\alpha}{2}}(h(x)-h(y))dy\\
&-\alpha C(\alpha)\frac{R'(x)}{ R(x)}\int \frac{\cos(x-y)(R(x)-R(y))}{\left((R(x)-R(y))^2+
4R(x)R(y)\sin\left(\frac{x-y}{2}\right)\right)^{\frac{\alpha}{2}+1}}\\ & \times \left((R(x)-R(y))(h(x)-h(y))+2(h(x)R(y)+h(y)R(x))\sin^2\left(\frac{x-y}{2}\right)\right)dy.
\end{align*}
Moreover, $D_i[R] h$ are continuous in $R$.
\end{lemma}

\begin{proof}
The proof of this lemma follows the same steps of the proof of Lemma \ref{gatoderivada} with similar modifications that the ones done in the proof of proposition \ref{prop2} with respect to the proof of proposition \ref{prop1}.
\end{proof}
\subsection{Steps 4, 5 and 6}
This is already done in the  previous section.
\end{proof}

\section{Regularity of solutions}

In this section, we will show that the solutions that we found and had limited regularity $X = \{X^{k}, X^{k+\alpha-1}, X^{k+\log}\}$ depending on whether $\al$ is smaller than, greater than or equal to 1 respectively) are indeed $C^{\infty}$. We will work with solutions that are contained in $B_r(1)$, the ball of radius $r$ around 1, which for simplicity we will denote by $B_r$. It will be clear from the context what norm to use in the different cases. To show the regularity, we will use the following common strategy in the three different cases $\al < 1$, $\al > 1$, and $\al = 1$:

 \begin{enumerate}
  \item Take $k-1$ derivatives and put the equation into the form
 \begin{align*}
  (LI+S)(\pa^{k-1} R)(x) = g(R)(x),
 \end{align*}
 where $LI$ is linear and invertible and $S$ satisfies $\|S(R_r)(x)\| \leq C_r\|R_r\|$, where $C_r  \rightarrow 0$ when $r \to 0$ for every $R_r \in B_r$. It is crucial that $C_r$ is bounded independently of $k$, since otherwise $B_r$ would shrink to 0 whenever we let $k$ go to infinity. 

 $LI+S$ will map functions from $H^{2-\al}$ into $H^{1-\al}$, $H^{2\al-1}$ into $H^{\al-1}$ or $X^{2+\log}$ into $H^{1}$ depending on $\al$ being smaller than, greater than or equal to 1 respectively and will be invertible since $C_r$ and $S$ can be as small as desired by taking $r$ small enough. We remark here that we are not inverting $R$ but $\pa^{k-1} R$ and we are regarding both $LI$ and $S$ as linear operators (that depend on $R$ and its lower order derivatives) acting on $\pa^{k-1} R$. 

 \item Show that if $R(x) \in X$, then $g(R)(x)$ is in $H^{\beta}$, where $\beta > 0$. This will allow us to bootstrap.
 \end{enumerate}

For simplicity, we will show how to bootstrap when $k$ is an integer. However, the proof can be adapted for the case $k \not \in \mathbb{Z}$.

 \subsection{The case $0 < \al < 1$}

 In this subsection, we will use $H^{k}$ spaces. There is no obstruction to the use of $C^{k}$ spaces, which is the space in which the existence of solutions was proved in \cite{Hassainia-Hmidi:v-states-generalized-sqg}.

We will choose the following $LI$, $S$ and $g(R)$. It is immediate to check that they satisfy equation \eqref{funcionalvstates} after taking $k-1$ derivatives:

\begin{align*}
  LI(\pa^{k-1} R)(x) & = \Omega (\pa^{k-1} R)'(x) + (\pa^{k-1} R)'(x) C(\al) \int \frac{\cos(y)}{\left(4\sin^2\left(\frac{y}{2}\right)\right)^{\al/2}}dy \\
S(\pa^{k-1} R)(x) & = -C(\al) \frac{(\pa^{k-1} R)'(x)}{R(x)} \int \frac{\cos(y)(R(x-y)-R(x))+\sin(y)R'(x-y)}{\left((R(x)-R(x-y))^2+4R(x)R(x-y)\sin^{2}\left(\frac{y}{2}\right)\right)^{\al/2}}dy \\
& + (\pa^{k-1} R)'(x) C(\al) \int \frac{\cos(y)}{\left(4\sin^2\left(\frac{y}{2}\right)\right)^{\al/2}}\left[\frac{1}{\left(\left(\frac{R(x)-R(x-y)}{2\sin\left(\frac{y}{2}\right)}\right)^{2} + R(x)R(x-y)\right)^{\al/2}}-1\right]dy \\
& = S_1(\pa^{k-1} R) + S_2(\pa^{k-1} R) \\
g(R) & = \pa^{k-1} \left(C(\al)\int \frac{\sin(y)R(x-y) + \cos(y)R'(x-y)}{\left((R(x)-R(x-y))^2+4R(x)R(x-y)\sin^{2}\left(\frac{y}{2}\right)\right)^{\al/2}}dy\right) \\
& + \pa^{k-1}\left(C(\al) \frac{R'(x)}{R(x)} \int \frac{\cos(y)(R(x-y)-R(x))+\sin(y)R'(x-y)}{\left((R(x)-R(x-y))^2+4R(x)R(x-y)\sin^{2}\left(\frac{y}{2}\right)\right)^{\al/2}}dy\right) \\
&  -C(\al) \frac{(\pa^{k-1} R)'(x)}{R(x)} \int \frac{\cos(y)(R(x-y)-R(x))+\sin(y)R'(x-y)}{\left((R(x)-R(x-y))^2+4R(x)R(x-y)\sin^{2}\left(\frac{y}{2}\right)\right)^{\al/2}}dy \\
& - \left(\pa^{k-1}\left(R'(x) C(\al) \int \frac{\cos(y)}{\left(4\sin^2\left(\frac{y}{2}\right)\right)^{\al/2}}\left[\frac{1}{\left(\left(\frac{R(x)-R(x-y)}{2\sin\left(\frac{y}{2}\right)}\right)^{2} + R(x)R(x-y)\right)^{\al/2}}-1\right]dy \right)\right. \\
& \left.- (\pa^{k-1} R)'(x) C(\al) \int \frac{\cos(y)}{\left(4\sin^2\left(\frac{y}{2}\right)\right)^{\al/2}}\left[\frac{1}{\left(\left(\frac{R(x)-R(x-y)}{2\sin\left(\frac{y}{2}\right)}\right)^{2} + R(x)R(x-y)\right)^{\al/2}}-1\right]dy\right) \\
& = G_1 + G_2 + G_3
\end{align*}

\begin{lemma}
\label{lemmaISalphamenor1}
$LI$ and $S$ satisfy the following properties:

\begin{enumerate}
 \item $LI$ is linear and invertible, and maps $H^{2-\al}$ into $H^{1-\al}$.
\item $\|S(\pa^{k-1} R_r)(x)\|_{H^{1-\al}} \leq C_r\|\pa^{k-1} R_r\|_{H^{2-\al}}$, where $C_r  \rightarrow 0$ when $r \to 0$ for every $R_r \in B_r$ and $C_r$ is independent of $k$.
\end{enumerate}

\end{lemma}

\begin{proof}
\begin{enumerate}
 \item The linearity of $LI$ is trivial. To check that $LI$ is invertible, we first compute the following integral using formula  \eqref{integralgammas}:

\begin{align*}
 \int \frac{\cos(y)}{(4\sin^2\left(\frac{y}{2}\right))^{\al/2}}dy
= -2\pi \frac{\Gamma(1-\al)}{\Gamma\left(2-\frac{\al}{2}\right)\Gamma\left(-\frac{\al}{2}\right)}.
\end{align*}

Therefore

\begin{align*}
 C(\al)\int \frac{\cos(y)}{(4\sin^2\left(\frac{y}{2}\right))^{\al/2}}dy
& = - 2^{\al-1}\frac{\Gamma(1-\al)\Gamma\left(\frac{\al}{2}\right)}{\Gamma\left(2-\frac{\al}{2}\right)\Gamma\left(-\frac{\al}{2}\right)\Gamma\left(1-\frac{\al}{2}\right)} \\
& = -2^{\al-1}\frac{\Gamma(1-\al)}{\left(\Gamma\left(1-\frac{\al}{2}\right)\right)^{2}}\left(\frac{\Gamma\left(\frac{\al}{2}\right)\left(-\frac{\al}{2}\right)}{\Gamma\left(1-\frac{\al}{2}\right)\left(1-\frac{\al}{2}\right)}\right)
= -2^{\al-1}\frac{\Gamma(1-\al)}{\left(\Gamma\left(1-\frac{\al}{2}\right)\right)^{2}}\left(-\frac{\Gamma\left(1+\frac{\al}{2}\right)}{\Gamma\left(2-\frac{\al}{2}\right)}\right),
\end{align*}

which implies

\begin{align*}
 \Omega_m +  C(\al)\int \frac{\cos(y)}{(4\sin^2\left(\frac{y}{2}\right))^{\al/2}}dy
& =2^{\al-1}\frac{\Gamma(1-\al)}{\left(\Gamma\left(1-\frac{\al}{2}\right)\right)^{2}}\left(\frac{\Gamma\left(m+\frac{\al}{2}\right)}{\Gamma\left(1+m-\frac{\al}{2}\right)}\right),
\end{align*}

which is different than zero for any real $m$. We remark that the possible values that $\Omega$ can take are neighborhoods of $\Omega_m$, i.e. neighbourhoods of integers $m$ in the previous formula, but the multiplier is non-zero for any value of $m$, which is a stronger statement. We can conclude the invertibility of $LI$ from this.

 \item 

We will use the following estimate:

\begin{lemma}
\label{lemmafolland}
 Let $s>0$, $\sigma > \frac12$ and let $\phi \in H^{s+\sigma} \cap L^{\infty}$, $f \in H^{s}$. Then:

\begin{align*}
 \|\phi f\|_{H^{s}} \leq C(\|\phi\|_{L^{\infty}}\|f\|_{H^{s}} + \|\phi\|_{H^{s+\sigma}}\|f\|_{L^{2}}).
\end{align*}

\end{lemma}

If we take $H^{1-\al}$ norms of $S_1(\pa^{k-1} R)$,

\begin{align*}
\|S_1\|_{H^{1-\al}} \leq C \|\pa^{k-1} R\|_{H^{2-\al}} \|R'\|_{L^{\infty}} + C_{r} \|\pa^{k-1} R\|_{H^{1}} = C_r \|\pa^{k-1} R\|_{H^{2-\al}},
\end{align*}

where we have used Lemma \ref{lemmafolland} and

\begin{align*}
 \left\|C(\al) \frac{1}{R(x)} \int \frac{\cos(y)(R(x-y)-R(x))+\sin(y)R'(x-y)}{\left((R(x)-R(x-y))^2+4R(x)R(x-y)\sin^{2}\left(\frac{y}{2}\right)\right)^{\al/2}}dy\right\|_{H^{1-\al+\sigma}} \leq C_r,
\end{align*}

where $\sigma > \frac12$. This shows the boundedness of $S_1$.

To bound $S_2$, we will use that for any $p,q > 0$:

\begin{align*}
 \left|\frac{1}{p^{\al}} - \frac{1}{q^{\al}}\right| \leq C_{\al}|q-p|\frac{1}{p^{\al}q^{\al}(p^{1-\al} + q^{1-\al})}.
\end{align*}

If we take

\begin{align*}
 p =\left(\left(\frac{R(x)-R(x-y)}{2\sin\left(\frac{y}{2}\right)}\right)^{2} + R(x)R(x-y)\right) , \quad q = 1.
\end{align*}

Then

\begin{align*}
 |p-q| \leq \left(\frac{R(x)-R(x-y)}{2\sin\left(\frac{y}{2}\right)}\right)^{2} + |(R(x)-1)R(x-y)| + |R(x-y)-1|,
\end{align*}

and the rest of the factors have upper and lower bounds. Thus, using again Lemma \ref{lemmafolland} with $\sigma > \frac12$:

\begin{align*}
 \|S_2\|_{H^{1-\al}} & \leq C \|\pa^{k-1} R\|_{H^{2-\al}} \left\|\frac{1}{\left(\left(\frac{R(x)-R(x-y)}{2\sin\left(\frac{y}{2}\right)}\right)^{2} + R(x)R(x-y)\right)^{\al/2}}-1\right\|_{L^{\infty}(x,y)} \\
& + C \|\pa^{k-1} R\|_{H^{1}} \left\|\int \frac{\cos(y)}{\left(4\sin^{2}\left(\frac{y}{2}\right)\right)^{\al/2}}\left(\frac{1}{\left(\left(\frac{R(x)-R(x-y)}{2\sin\left(\frac{y}{2}\right)}\right)^{2} + R(x)R(x-y)\right)^{\al/2}}-1\right)\right\|_{H^{1-\al+\sigma}} \\
& \leq C \|\pa^{k-1} R\|_{H^{2-\al}} (\|R'\|_{L^{\infty}}^{2} + (\|R\|_{L^{\infty}}+1)\|R-1\|_{L^{\infty}}) + C_{r} \|\pa^{k-1} R\|_{H^{1}} \\
& = C_r \|\pa^{k-1} R\|_{H^{2-\al}}.
\end{align*}

This finishes the $H^{1-\al}$-boundedness of $S$.

\end{enumerate}

\end{proof}

\begin{lemma}
\label{lemmaGalphamenor1}
Let $g(R)$ be defined as before, and let $R \in H^{k}$. Then $g(R) \in H^{1-\al}$.
\end{lemma}

\begin{proof}
 \begin{enumerate}
  \item To bound $G_1$ in $H^{1-\al}$, we first notice that the part of $G_1$ that contains the factor $\sin(y)R(x-y)$ is of lower order term than the one with $\cos(y)R'(x-y)$. Hence, we will focus on the latter. We first apply the $k-1$ derivatives and look to the most singular terms. One of them is the following:

\begin{align*}
\left(C(\al)\int \frac{\cos(y)\pa^{k} R(x-y)}{\left((R(x)-R(x-y))^2+4R(x)R(x-y)\sin^{2}\left(\frac{y}{2}\right)\right)^{\al/2}}dy\right).
\end{align*}

We now integrate by parts to obtain

\begin{align*}
& \al C(\al)\int \frac{\cos(y)(\pa^{k-1}R(x) - \pa^{k-1} R(x-y))}{\left((R(x)-R(x-y))^2+4R(x)R(x-y)\sin^{2}\left(\frac{y}{2}\right)\right)^{\al/2+1}} \\
& \times \left((R(x)-R(x-y))R'(x-y) + 2R(x)R(x-y)\sin\left(\frac{y}{2}\right)\cos\left(\frac{y}{2}\right)\right)dy + \text{l.o.t}.
\end{align*}

We can now split the kernel

\begin{align*}
 & \frac{\cos(y)\left((R(x)-R(x-y))R'(x-y) + 2R(x)R(x-y)\sin\left(\frac{y}{2}\right)\cos\left(\frac{y}{2}\right)\right)}{\left((R(x)-R(x-y))^2+4R(x)R(x-y)\sin^{2}\left(\frac{y}{2}\right)\right)^{\al/2+1}} \\
& = \frac{1}{\left((R(x))^2 + (R'(x))^2\right)^{\al/2}} \frac{1}{\left|\sin\left(\frac{y}{2}\right)\right|^{\al+1}} + \text{Rem}(x,y),
\end{align*}

where Rem$(x,y) \in \mathcal{H}_{0}$. Plugging the decomposition of the kernel, we need to bound the $H^{1-\al}$ of

\begin{align*}
& \al C(\al)\frac{1}{\left((R(x))^2 + (R'(x))^2\right)^{\al/2}}\int \frac{(\pa^{k-1}R(x) - \pa^{k-1} R(x-y))}{\left|\sin\left(\frac{y}{2}\right)\right|^{\al+1}}dy \\
& + \al C(\al)\int (\pa^{k-1}R(x) - \pa^{k-1} R(x-y))\text{Rem}(x,y)dy \\
& = G_{11} + G_{12}.
\end{align*}

Taking $\Lambda^{1-\al}$, modulo low order commutators, the most singular terms are

\begin{align*}
& \al C(\al)\frac{1}{\left((R(x))^2 + (R'(x))^2\right)^{\al/2}}\int \frac{(\Lambda^{1-\al}\pa^{k-1}R(x) - \Lambda^{1-\al}\pa^{k-1} R(x-y))}{\left|\sin\left(\frac{y}{2}\right)\right|^{\al+1}}dy \\
& + \al C(\al)\int (\Lambda^{1-\al}\pa^{k-1}R(x) - \Lambda^{1-\al}\pa^{k-1} R(x-y))\text{Rem}(x,y)dy, \\
\end{align*}

which can be bounded by $C\|R\|_{H^{k-1+1-\al+\al+1-1}} = C\|R\|_{H^{k}}$ and $C\|R\|_{H^{k-\al}}$ (low order) respectively.

Finally, the term where 1 of the $k-1$ derivatives hits the denominator and the other $k-2$ hit the factor $R'(x)-R'(x-y)$ is treated the same way. This concludes the boundedness of $G_1$.

\item We start calculating the most singular terms of $G_2$. First, if all the derivatives hit the same factor on the numerator:

\begin{align*}
 & C(\al) \frac{R'(x)}{R(x)} \int \frac{\cos(y)(\pa^{k-1} R(x-y)-\pa^{k-1} R(x))+\sin(y)\pa^{k} R(x-y)}{\left((R(x)-R(x-y))^2+4R(x)R(x-y)\sin^{2}\left(\frac{y}{2}\right)\right)^{\al/2}}dy \\
& = G_{21} + G_{22}.
\end{align*}

$G_{21}$ is less singular than $G_{11}$ and it is estimated in the same way. To deal with $G_{22}$, we integrate by parts and get a kernel that is of the same order as $G_{21}$, therefore repeating the same procedure as with $G_{11}$. The final factor is the one we get if we hit the denominator with one derivative, and the $k-2$ remaining ones hit the factor $R'(x) - R'(x-y)$, but again this is treated as $G_{21}$.

\item We compute the most singular terms of $G_{3}$, obtaining

\begin{align*}
&\left(\frac{\al}{2}\right) R'(x) C(\al) \int \frac{\cos(y)}{\left(4\sin^2\left(\frac{y}{2}\right)\right)^{\al/2}}\frac{1}{\left(\left(\frac{R(x)-R(x-y)}{2\sin\left(\frac{y}{2}\right)}\right)^{2} + R(x)R(x-y)\right)^{\al/2+1}} \\
&\times \left(\frac{(R(x)-R(x-y))(\pa^{k-1}R(x) - \pa^{k-1}R(x-y))}{4\sin^2\left(\frac{y}{2}\right)}\right)dy,  \\
\end{align*}

which is easily bounded like $G_{1}$.

 \end{enumerate}

\end{proof}

Combining the two previous lemmas, we obtain the following corollary:

\begin{corollary}
 Let $R \in H^{k}$. Then $R \in H^{k+(1-\al)}$.
\label{bootstrapalphamenor1}
\end{corollary}
\begin{proof}
 By Lemma \ref{lemmaGalphamenor1}, $g(R) \in H^{k-\al}$, and by Lemma \ref{lemmaISalphamenor1}, $LI+S$ is invertible. $(LI+S)^{-1}$ maps $H^{1-\al}$ into $H^{2-\al}$. Thus

\begin{align*}
 \pa^{k-1} R = \underbrace{(LI+S)^{-1}\underbrace{g(R)}_{\in H^{1-\al}}}_{\in H^{2-\al}} \in H^{2-\al},
\end{align*}

which implies $R \in H^{k+1-\al}$.

\end{proof}

\subsection{The case $1 < \al < 2$}

We will choose the following $LI$, $S$ and $g(R)$. It is immediate to check that they satisfy equation \eqref{funcionalvstates} after taking $k-1$ derivatives:

 \begin{align*}
  LI(\pa^{k-1} R)(x) & = -C(\al) \int \frac{(\pa^{k-1} R)'(x-y)-(\pa^{k-1} R)'(x)}{\left(4\sin^2\left(\frac{y}{2}\right)\right)^{\al/2}}dy \\
 S(\pa^{k-1} R)(x) & = -C(\al) \int \frac{((\pa^{k-1} R)'(x-y)-(\pa^{k-1} R)'(x))}{\left(4\sin^2\left(\frac{y}{2}\right)\right)^{\al/2}}\left[\frac{1}{\left(\left(R'(x)\right)^{2} + (R(x))^2\right)^{\al/2}}-1\right]dy \\
 g(R)(x) & = -\Omega \pa^{k} R(x) \\
 & + \pa^{k-1} \left(C(\al)\frac{R'(x)}{R(x)} \int \frac{\cos(y)(R(x-y)-R(x))+\sin(y)R'(x-y)}{\left((R(x)-R(x-y))^2+4R(x)R(x-y)\sin^{2}\left(\frac{y}{2}\right)\right)^{\al/2}}dy\right) \\
 & + \pa^{k-1} \left(C(\al)\int \frac{\sin(y)R(x-y)}{\left((R(x)-R(x-y))^2+4R(x)R(x-y)\sin^{2}\left(\frac{y}{2}\right)\right)^{\al/2}}dy\right) \\
& + \left(\pa^{k-1} \left(C(\al) \int \frac{\cos(y)(R'(x-y)-R'(x))}{\left(4\sin^2\left(\frac{y}{2}\right)\right)^{\al/2}}\left[\frac{1}{\left(\left(\frac{R(x)-R(x-y)}{2\sin\left(\frac{y}{2}\right)}\right)^{2} + R(x)R(x-y)\right)^{\al/2}}\right]dy\right)\right. \\
& \left.- C(\al) \int \frac{\cos(y)((\pa^{k-1} R)'(x-y)-(\pa^{k-1} R)'(x))}{\left(4\sin^2\left(\frac{y}{2}\right)\right)^{\al/2}}\left[\frac{1}{\left(\left(\frac{R(x)-R(x-y)}{2\sin\left(\frac{y}{2}\right)}\right)^{2} + R(x)R(x-y)\right)^{\al/2}}\right]dy\right. \\
& \left. + C(\al) \int \frac{((\pa^{k-1} R)'(x-y)-(\pa^{k-1} R)'(x))}{\left(4\sin^2\left(\frac{y}{2}\right)\right)^{\al/2}}\right. \\
& \left.\times \left[\frac{\cos(y)}{\left(\left(\frac{R(x)-R(x-y)}{2\sin\left(\frac{y}{2}\right)}\right)^{2} + R(x)R(x-y)\right)^{\al/2}}-\frac{1}{((R'(x))^2 + (R(x))^2)^{\al/2}}\right]dy\right)\\
& = G_1(R) + G_2(R) + G_3(R) + G_4(R) + G_5(R)
 \end{align*}

\begin{lemma}
\label{lemmaISalphamayor1}
$LI$ and $S$ satisfy the following properties:

\begin{enumerate}
 \item $LI$ is linear and invertible, and maps $H^{2\al-1}$ into $H^{\al-1}$.
\item $\|S(\pa^{k-1} R_r)(x)\|_{H^{\al-1}} \leq C_r\|\pa^{k-1} R_r\|_{H^{2\al-1}}$, where $C_r  \rightarrow 0$ when $r \to 0$ for every $R_r \in B_r$ and $C_r$ is independent of $k$.
\end{enumerate}

\end{lemma}

\begin{proof}
\begin{enumerate}
 \item The linearity of $LI$ is trivial. We saw in section \ref{sectionexistencealpha1} that $LI$ is invertible since the multiplier in Fourier space does not vanish and moreover it maps $H^{2\al-1}$ into $H^{\al-1}$ by definition of the $H^{s}$-norm.
\item  We note that  we can bound, for $\sigma > \frac12$:

\begin{align*}
 \left\|\frac{1}{\left(\left(R'(x)\right)^{2} + (R(x))^{2}\right)^{\al/2}}-1\right\|_{L^{\infty}} & \leq C\left\|(R'(x))^{2} + (R(x))^{2} - 1\right\|_{L^{\infty}}
\leq C \left\|1-R\right\|_{L^{\infty}} + C\|R'\|_{L^{\infty}} = C_r. \\
 \left\|\frac{1}{\left(\left(R'(x)\right)^{2} + (R(x))^{2}\right)^{\al/2}}-1\right\|_{H^{\al-1+\sigma}} & \leq C_{r}. \\
\end{align*}

Then, $S$ can be bounded in the $H^{\al-1}$-norm (for $\sigma > \frac12$) by means of Lemma \ref{lemmafolland} by the sum of the following two terms:

\begin{align*}
 C \left\|\int \frac{(\pa^{k} R(x-y)- \pa^{k} R(x))}{\left(4\sin^2\left(\frac{y}{2}\right)\right)^{\al/2}}dy\right\|_{H^{\al-1}}\left\|\frac{1}{\left(\left(R'(x)\right)^{2} + (R(x))^{2}\right)^{\al/2}}-1\right\|_{L^{\infty}} \leq  C_r \|\pa^{k-1} R\|_{H^{2\al-1}} \\
\end{align*}

and 
\begin{align*}
 C \left\|\int \frac{(\pa^{k} R(x-y)- \pa^{k} R(x))}{\left(4\sin^2\left(\frac{y}{2}\right)\right)^{\al/2}}dy\right\|_{L^{2}}\left\|\frac{1}{\left(\left(R'(x)\right)^{2} + (R(x))^{2}\right)^{\al/2}}\right\|_{H^{\al-1+\sigma}} \leq  C_r \|\pa^{k-1} R\|_{H^{\al}},
\end{align*}

therefore getting the bound

\begin{align*}
 \|S\|_{H^{\al-1}} \leq C_{r} \|\pa^{k-1} R\|_{H^{2\al-1}}
\end{align*}

\end{enumerate}

\end{proof}

\begin{lemma}
\label{lemmaGalphamayor1}
Let $g(R)$ be defined as before, and let $R \in H^{k+\al-1}$. Then $g(R) \in H^{\al-1}$.
\end{lemma}

\begin{proof}
 We will go term by term over the $G_i$.
\begin{enumerate}
 \item $G_1$ is trivial and the $H^{\al-1}$-norm is clearly bounded by a constant times the $H^{k+\al-1}$-norm of $R$.
\item We apply the $k-1$ derivatives and look for the most singular terms. These are:

\begin{align*}
& C(\al)\frac{\pa^{k} R(x)}{R(x)} \int \frac{\cos(y)(R(x-y)-R(x))+\sin(y)R'(x-y)}{\left((R(x)-R(x-y))^2+4R(x)R(x-y)\sin^{2}\left(\frac{y}{2}\right)\right)^{\al/2}}dy \\
& + C(\al)\frac{R'(x)}{R(x)} \int \frac{\cos(y)(\pa^{k-1} R(x-y)-\pa^{k-1} R(x))+\sin(y)\pa^{k} R(x-y)}{\left((R(x)-R(x-y))^2+4R(x)R(x-y)\sin^{2}\left(\frac{y}{2}\right)\right)^{\al/2}}dy \\
& + C(\al)\left(-\frac{\al}{2}\right)\frac{R'(x)}{R(x)} \int \frac{\cos(y)(R(x-y)-R(x))+\sin(y)R'(x-y)}{\left((R(x)-R(x-y))^2+4R(x)R(x-y)\sin^{2}\left(\frac{y}{2}\right)\right)^{\al/2+1}} \\
& \times \left(2(R(x)-R(x-y))(\pa^{k-1}R(x) - \pa^{k-1}R(x-y))+4(\pa^{k-1}R(x)R(x-y) + R(x)\pa^{k-1}R(x-y))\sin^{2}\left(\frac{y}{2}\right)\right) dy \\
& = G_{21} + G_{22} + G_{23}
\end{align*}

The $H^{\al-1}$ norm  of $G_{21}$ can be bounded by

\begin{align*}
 \|G_{21}\|_{H^{\al-1}} \leq C\|R\|_{H^{k+\al-1}} + \text{l.o.t},
\end{align*}

We split $G_{22}$ into two:

\begin{align*}
&  C(\al)\frac{R'(x)}{R(x)} \int \frac{\cos(y)(\pa^{k-1} R(x-y)-\pa^{k-1} R(x))}{\left((R(x)-R(x-y))^2+4R(x)R(x-y)\sin^{2}\left(\frac{y}{2}\right)\right)^{\al/2}}dy \\
& + C(\al)\frac{R'(x)}{R(x)} \int \frac{\sin(y)\pa^{k} R(x-y)}{\left((R(x)-R(x-y))^2+4R(x)R(x-y)\sin^{2}\left(\frac{y}{2}\right)\right)^{\al/2}}dy \\
& = G_{221} + G_{222}
\end{align*}

We deal first with $G_{222}$. We take first $\Lambda^{\al-1}$ and estimate its $L^{2}$ norm. Modulo low order commutators, the most singular term will be

\begin{align*}
 C(\al)\frac{R'(x)}{R(x)} \int \frac{\sin(y)\Lambda^{\al-1} \pa^{k} R(x-y)}{\left((R(x)-R(x-y))^2+4R(x)R(x-y)\sin^{2}\left(\frac{y}{2}\right)\right)^{\al/2}}dy,
\end{align*}

which can be bounded in $L^{2}$ by a multiple of $\|R\|_{H^{k+\al-1}}$ since the kernel

\begin{align*}
 K(x,y) = \frac{\sin(y)}{\left((R(x)-R(x-y))^2+4R(x)R(x-y)\sin^{2}\left(\frac{y}{2}\right)\right)^{\al/2}} \in \mathcal{H}_{0}.
\end{align*}

We move on to $G_{221}$. We first decompose the kernel into

\begin{align*}
 \frac{\cos(y)}{\left((R(x)-R(x-y))^2+4R(x)R(x-y)\sin^{2}\left(\frac{y}{2}\right)\right)^{\al/2}}
= \frac{1}{\left(\left((R'(x)\right)^2 + (R(x))^2\right)^{\al/2}} \frac{1}{\left(4 \sin^{2}\left(\frac{y}{2}\right)\right)^{\al/2}} + \text{Rem}(x,y),
\end{align*}

where Rem$(x,y) \in \mathcal{H}_{1-\al}$. Again, modulo commutators and taking $\al-1$ derivatives, we can write the most singular terms as

\begin{align*}
  & C(\al)\frac{R'(x)}{R(x)}\frac{1}{\left(\left((R'(x)\right)^2 + (R(x))^2\right)^{\al/2}}\int \frac{(\Lambda^{\al-1}\pa^{k-1} R(x-y)-\Lambda^{\al-1}\pa^{k-1} R(x))}{\left(4 \sin^{2}\left(\frac{y}{2}\right)\right)^{\al/2}}dy \\
 & + C(\al)\frac{R'(x)}{R(x)} \int (\Lambda^{\al-1}\pa^{k-1} R(x-y)-\Lambda^{\al-1}\pa^{k-1} R(x))\text{Rem}(x,y)dy \\
\end{align*}

Thus

\begin{align*}
 \|G_{221}\|_{H^{\al}} \leq C\|R\|_{H^{k-3+2\al}}  + \text{l.o.t} \leq C\|R\|_{H^{k+\al-1}} + \text{l.o.t}
\end{align*}

The terms that contain a factor of $4(\pa^{k-1}R(x)R(x-y) + R(x)\pa^{k-1}R(x-y))\sin^{2}\left(\frac{y}{2}\right)$ in $G_{23}$ are low order and can be estimated as the previous ones. From what is left, again, modulo low order commutators, the most singular term is

\begin{align*}
 & C(\al)\left(-\frac{\al}{2}\right)\frac{R'(x)}{R(x)} \int \frac{\cos(y)(R(x-y)-R(x))+\sin(y)R'(x-y)}{\left((R(x)-R(x-y))^2+4R(x)R(x-y)\sin^{2}\left(\frac{y}{2}\right)\right)^{\al/2+1}} \\
& \times(2(R(x)-R(x-y))(\Lambda^{\al-1} \pa^{k-1}R(x) - \Lambda^{\al-1}\pa^{k-1}R(x-y))) dy, \\
\end{align*}

but this term can be bounded as $G_{221}$. This finishes $G_{2}$.

%
%
%
%
%
%

\item $G_3$ is less singular than $G_2$ and thus can be bounded in the same way.
\item The most singular terms in $G_4$ are when $k-2$ derivatives hit the factor $(R'(x-y)-R'(x))$ and one hits the denominator, or when one hits the denominator and the remaining $k-2$ hit one of the factors $(R'(x-y)-R'(x))$. Both terms can be estimated as $G_{221}$.
\item $G_5$ can also be estimated as $G_{221}$ using that the kernel

\begin{align*}
\left[\frac{\cos(y)}{\left(\left(\frac{R(x)-R(x-y)}{2\sin\left(\frac{y}{2}\right)}\right)^{2} + R(x)R(x-y)\right)^{\al/2}}-\frac{1}{((R'(x))^2 + (R(x))^2)^{\al/2}}\right]
\end{align*}

belongs to $\mathcal{H}_{1}$.
\end{enumerate}

\end{proof}

Combining the two previous lemmas, we obtain the following corollary:

\begin{corollary}
\label{bootstrapalphamayor1}
 Let $R \in H^{k+\al-1}$. Then $R \in H^{k+2\al-2}$.
\end{corollary}
\begin{proof}
 By Lemma \ref{lemmaGalphamayor1}, $g(R) \in H^{\al-1}$, and by Lemma \ref{lemmaISalphamayor1}, $LI+S$ is invertible. $(LI+S)^{-1}$ maps $H^{\al-1}$ into $H^{2\al-1}$. Thus

\begin{align*}
 \pa^{k-1} R = \underbrace{(LI+S)^{-1}\underbrace{g(R)}_{\in H^{\al-1}}}_{\in H^{2\al-1}} \in H^{2\al-1},
\end{align*}

which implies $R \in H^{k+2\al-2}$.

\end{proof}

\subsection{The case $\al = 1$}

We will choose the following $LI$, $S$ and $g(R)$. It is immediate to check that they satisfy equation \eqref{funcionalvstates} after taking $k-1$ derivatives:

\begin{align*}
 LI(\pa^{k-1} R)(x) & = \Omega (\pa^{k-1} R)'(x) - \frac{1}{2\pi} \int \frac{\cos(y)((\pa^{k-1} R)'(x-y)-(\pa^{k-1} R)'(x)}{\left(4\sin^2\left(\frac{y}{2}\right)\right)^{1/2}}dy \\
 S(\pa^{k-1} R)(x) & = -\frac{1}{2\pi}\frac{(\pa^{k-1} R)'(x)}{R(x)} \int \frac{\cos(y)(R(x)-R(x-y)) + \sin(y)R'(x-y)}{\left((R(x)-R(x-y))^2+4R(x)R(x-y)\sin^{2}\left(\frac{y}{2}\right)\right)^{1/2}}dy \\
& -\frac{1}{2\pi} \int \frac{\cos(y)((\pa^{k-1} R)'(x-y)-(\pa^{k-1} R)'(x))}{\left(4\sin^2\left(\frac{y}{2}\right)\right)^{1/2}}\left[\frac{1}{\left(\left(\frac{R(x)-R(x-y)}{2\sin\left(\frac{y}{2}\right)}\right)^{2} + R(x)R(x-y)\right)^{1/2}}-1\right]dy \\
& = S_1 + S_2 \\
g(R)(x) & = \pa^{k-1}\left( \frac{1}{2\pi}\int \frac{\sin(y)R(x-y)}{\left((R(x)-R(x-y))^2+4R(x)R(x-y)\sin^{2}\left(\frac{y}{2}\right)\right)^{1/2}}dy\right) \\
& + \left(\pa^{k-1}\left(\frac{1}{2\pi}\frac{R'(x)}{R(x)} \int \frac{\cos(y)(R(x)-R(x-y)) + \sin(y)R'(x-y)}{\left((R(x)-R(x-y))^2+4R(x)R(x-y)\sin^{2}\left(\frac{y}{2}\right)\right)^{1/2}}dy\right)\right. \\
& \left. - \frac{1}{2\pi}\frac{(\pa^{k-1} R)'(x)}{R(x)} \int \frac{\cos(y)(R(x)-R(x-y)) + \sin(y)R'(x-y)}{\left((R(x)-R(x-y))^2+4R(x)R(x-y)\sin^{2}\left(\frac{y}{2}\right)\right)^{1/2}}dy \right) \\
 & + \left(\pa^{k-1} \left(\frac{1}{2\pi} \int \frac{\cos(y)(R'(x-y)-R'(x))}{\left(4\sin^2\left(\frac{y}{2}\right)\right)^{1/2}}\left[\frac{1}{\left(\left(\frac{R(x)-R(x-y)}{2\sin\left(\frac{y}{2}\right)}\right)^{2} + R(x)R(x-y)\right)^{1/2}}\right]dy\right)\right. \\
 & \left.- \frac{1}{2\pi} \int \frac{\cos(y)((\pa^{k-1} R)'(x-y)-(\pa^{k-1} R)'(x))}{\left(4\sin^2\left(\frac{y}{2}\right)\right)^{1/2}}\left[\frac{1}{\left(\left(\frac{R(x)-R(x-y)}{2\sin\left(\frac{y}{2}\right)}\right)^{2} + R(x)R(x-y)\right)^{1/2}}\right]dy\right. \\
& = G_1 + G_2 + G_3
\end{align*}

\begin{lemma}
\label{lemmaISalpha1}

$LI$ and $S$ satisfy the following properties:

\begin{enumerate}
 \item For every $\Omega$, $LI$ is linear, and maps $X^{2+\log}$ into $H^{1}$. $LI(\Omega_{m})$ is not invertible, and has a one-dimensional kernel spanned by $\cos(mx)$.
\item $\|S(\pa^{k-1} R_r)(x)\|_{H^{1}} \leq C_r\|\pa^{k-1} R_r\|_{X^{2+\log}}$, where $C_r  \rightarrow 0$ when $r \to 0$ for every $R_r \in B_r$ and $C_r$ is independent of $k$.
\end{enumerate}
\end{lemma}
\begin{proof}
\begin{enumerate}
 \item The linearity of $LI$ is trivial. We saw in section \ref{sectionexistencealpha1} that $LI(\Omega_{m})$ is not invertible, its kernel is spanned by $\cos(mx)$ and $LI$ maps $X^{2+\log}$ into $H^{1}$. This creates a technical problem, since the inverse of $LI$ is not uniformly bounded in $\Omega$. We deal with this problem in Corollary \ref{bootstrapalpha1}.
 \item  We start with $S_2$ and decompose the kernel

\begin{align*}
\frac{1}{\left(\left(\frac{R(x)-R(x-y)}{2\sin\left(\frac{y}{2}\right)}\right)^{2} + R(x)R(x-y)\right)^{1/2}}
= \frac{1}{\left((R(x))^{2} + (R'(x))^{2}\right)^{1/2}} + \text{Rem}(x,y),
\end{align*}

where Rem$(x,y) \in \mathcal{H}_{1}$ and

\begin{align*}
\sup_{x\in \T} \left\|\frac{\text{Rem}(x,\cdot)}{\sin(\cdot)}\right\|_{L^1(\T)}\leq C_r, \quad  \sup_{y \in \T}\left\|\frac{\text{Rem}(\cdot,y)}{\sin(y)}\right\|_{L^1(\T)}\leq C_r \\
\sup_{x\in \T} \left\|\frac{\pa_{x} \text{Rem}(x,\cdot)}{\sin(\cdot)}\right\|_{L^1(\T)}\leq C_r, \quad  \sup_{y \in \T}\left\|\frac{\pa_{x} \text{Rem}(\cdot,y)}{\sin(y)}\right\|_{L^1(\T)}\leq C_r \\
\end{align*}

%

This means that $S_2$ can be written as

\begin{align*}
 S_2 & = -\frac{1}{2\pi} \int \frac{((\pa^{k-1} R)'(x-y)-(\pa^{k-1} R)'(x))}{\left(4\sin^2\left(\frac{y}{2}\right)\right)^{1/2}}\left[\frac{1}{\left(\left(R'(x)\right)^{2} + (R(x))^2\right)^{1/2}}-1\right]dy \\
& -\frac{1}{2\pi} \int \frac{(\pa^{k-1} R)'(x-y)-(\pa^{k-1} R)'(x))}{\left(4\sin^{2}\left(\frac{y}{2}\right)\right)^{1/2}}\text{Rem}(x,y)dy.
\end{align*}

We note that  we can bound

 \begin{align*}
  \left\|\frac{1}{\left(\left(R'(x)\right)^{2} + (R(x))^{2}\right)^{1/2}}-1\right\|_{L^{\infty}} \leq C\left\|(R'(x))^{2} + (R(x))^{2} - 1\right\|_{L^{\infty}}
 \leq C \left\|1-R\right\|_{L^{\infty}} + C\|R'\|_{L^{\infty}} = C_r. \\
\left\|\pa_{x}\left(\frac{1}{\left(\left(R'(x)\right)^{2} + (R(x))^{2}\right)^{1/2}}-1\right)\right\|_{L^{\infty}} \leq C_r
 \end{align*}

 Then, $S_2$ can be bounded in the $H^{1}$-norm by

 \begin{align*}
  & C \left\|\pa^{k-1} R\right\|_{H^{2+\log}}\left\|\frac{1}{\left(\left(R'(x)\right)^{2} + (R(x))^{2}\right)^{1/2}}-1\right\|_{L^{\infty}} + C_r\|\pa^{k-1}R\|_{H^{1+\log}} + C_r\|\pa^{k-1} R\|_{H^{2+\log}} = C_r \|\pa^{k-1} R\|_{H^{2+\log}}.
 \end{align*}

In order to bound the $H^{1}$ norm of $S_1$, the most singular term can be bounded by 

\begin{align*}
 \|S_1\|_{H^{1}} \leq C \|\pa^{k-1} R\|_{H^{2}} \|R'\|_{L^{\infty}} + C \|\pa^{k-1} R\|_{H^{1}} \|R''\|_{L^{\infty}} \leq C_r \|\pa^{k-1} R\|_{H^{2+\log}}.
\end{align*}

\end{enumerate}
\end{proof}

\begin{lemma}
\label{lemmaGalpha1}
Let $g(R)$ be defined as before, and let $R \in X^{k+\log}$. Then $g(R) \in H^{1}$.
\end{lemma}

\begin{proof}
 \begin{enumerate}
\item We start with $G_{2}$. We take $k$ derivatives (the outer $k-1$ plus one more) and compute the most singular terms. The objective is to bound the terms in $L^{2}$ norm. The terms are

\begin{align*}
G_{21} & = \frac{1}{2\pi}\frac{R'(x)}{R(x)} \int \frac{\cos(y)(\pa^{k} R(x)-\pa^{k} R(x-y))}{\left((R(x)-R(x-y))^2+4R(x)R(x-y)\sin^{2}\left(\frac{y}{2}\right)\right)^{1/2}}dy \\
G_{22} & = \frac{1}{2\pi}\frac{R'(x)}{R(x)} \int \frac{\sin(y)\pa^{k+1} R(x-y)}{\left((R(x)-R(x-y))^2+4R(x)R(x-y)\sin^{2}\left(\frac{y}{2}\right)\right)^{1/2}}dy \\
G_{23} & = \frac{1}{4\pi}\frac{R'(x)}{R(x)} \int \frac{\cos(y)(R(x)-R(x-y)) + \sin(y)R'(x-y)}{\left((R(x)-R(x-y))^2+4R(x)R(x-y)\sin^{2}\left(\frac{y}{2}\right)\right)^{3/2}} \\
& \times \left(2(R(x)-R(x-y))(\pa^{k} R(x) - \pa^{k} R(x-y))+4(\pa^{k} R(x) R(x-y) + \pa^{k}R(x-y))\sin^{2}\left(\frac{y}{2}\right)\right)dy.
\end{align*}

We start with $G_{21}$. We split the kernel in the following way:

\begin{align*}
 \frac{\cos(y)}{\left((R(x)-R(x-y))^2+4R(x)R(x-y)\sin^{2}\left(\frac{y}{2}\right)\right)^{1/2}}
= \frac{1}{\left((R'(x))^2 + (R(x))^2\right)^{1/2}}\frac{1}{\left(4\sin^{2}\left(\frac{y}{2}\right)\right)^{1/2}} + \text{Rem}(x,y),
\end{align*}

where Rem$(x,y) \in \mathcal{H}_{0}$. $G_{21}$ becomes

\begin{align*}
 \frac{1}{2\pi}\frac{R'(x)}{R(x)} \int \frac{\pa^{k} R(x)-\pa^{k} R(x-y)}{\left((R'(x))^2 + (R(x))^2\right)^{1/2}}\frac{1}{\left(4\sin^{2}\left(\frac{y}{2}\right)\right)^{1/2}}dy + \frac{1}{2\pi}\frac{R'(x)}{R(x)} \int(\pa^{k} R(x)-\pa^{k} R(x-y))\text{Rem}(x,y)dy,
\end{align*}

which can be bounded in $L^{2}$ by

\begin{align*}
 C\|R\|_{X^{k+\log}} + C\|R\|_{H^{k}} < \infty.
\end{align*}

We integrate by parts in $G_{22}$ to get

\begin{align*}
 G_{22} & = -\frac{1}{2\pi}\frac{R'(x)}{R(x)} \int \frac{\pa^{k} R(x) - \pa^{k} R(x-y)}{\left((R(x)-R(x-y))^2+4R(x)R(x-y)\sin^{2}\left(\frac{y}{2}\right)\right)^{1/2}} \\
& \times \left(\cos(y) - \frac{\sin(y)}{2}\frac{2(R(x)-R(x-y)R'(x-y) - 4R(x)R'(x-y)\sin^{2}\left(\frac{y}{2}\right) +4R(x)R(x-y)\sin\left(\frac{y}{2}\right)\cos\left(\frac{y}{2}\right)}{\left((R(x)-R(x-y))^2+4R(x)R(x-y)\sin^{2}\left(\frac{y}{2}\right)\right)}\right)dy
\end{align*}

Again, the kernel can be decomposed into

\begin{align*}
 & \frac{1}{\left((R(x)-R(x-y))^2+4R(x)R(x-y)\sin^{2}\left(\frac{y}{2}\right)\right)^{1/2}} \\
& \times \left(\cos(y) - \frac{\sin(y)}{2}\frac{2(R(x)-R(x-y)R'(x-y) - 4R(x)R'(x-y)\sin^{2}\left(\frac{y}{2}\right) +4R(x)R(x-y)\sin\left(\frac{y}{2}\right)\cos\left(\frac{y}{2}\right)}{\left((R(x)-R(x-y))^2+4R(x)R(x-y)\sin^{2}\left(\frac{y}{2}\right)\right)}\right) \\
& = \frac{1}{\left(4\sin^{2}\left(\frac{y}{2}\right)\right)^{1/2}}\frac{R(x)R'(x) + R'(x)R''(x)}{\left((R(x))^2 + (R'(x))^2\right)^{2}} + \text{Rem}(x,y),
\end{align*}

where Rem$(x,y) \in \mathcal{H}_{0}$. This produces the following bound:

\begin{align*}
 \|G_{22}\|_{L^{2}} \leq C\|R\|_{X^{k+\log}} + \text{l.o.t}
\end{align*}

Finally, $G_{23}$ is estimated as $G_{21}$.

\item $G_1$ is less singular than $G_{22}$ and is estimated the same way.

\item The most singular terms in $G_{3}$ are the ones when we hit with 1 derivative the denominator and $k-1$ derivatives one of the $R'(x) - R'(x-y)$ factors. But they are estimated in the same way as $G_{21}$.
 \end{enumerate}

\end{proof}

Let us now decompose $R$ into $R^{high}$, the part corresponding to the frequencies greater than $m$, and $R^{low}$, the part corresponding to the frequencies smaller or equal than $m$. Since we can regard $LI$ and $S$ as linear operators acting on $\pa^{k-1} R$, we can write

\begin{align}
\label{ecuaciontildeISG}
 (LI + S)(\pa^{k-1} R^{high}) + (LI + S)(\pa^{k-1} R^{low}) = g(R) \\
 \Rightarrow
 (LI + S)(\pa^{k-1} R^{high})  = g(R) - (LI + S)(\pa^{k-1} R^{low}) \equiv \tilde{g}(R)
\end{align}

We do this splitting because we will want to make the norm of $S$ small with respect to $\frac{1}{\|LI^{-1}\|}$ to be able to invert $LI+S$. This may not be possible since both quantities ($\|S\|$ and $\frac{1}{\|LI^{-1}\|}$) go to zero as $\Omega \rightarrow \Omega_{m}$ and it is not clear that one is smaller than the other. We will prevent this situation from happening by inverting $LI$ only on high frequencies.

\begin{lemma}
Let $\tilde{g}(R)$ be defined as above, and let $R \in X^{k+\log}$. Then $\tilde{g}(R) \in H^{1}$.
\end{lemma}
\begin{proof}
 The regularity of $g(R)$ was proved in Lemma \ref{lemmaGalpha1}. We only need to show the regularity of $(LI+S)(\pa^{k-1} R^{low})$, but this follows easily from the fact that $R^{low}$ is analytic and therefore $LI$ and $S$ can be bounded by low order norms of $R$.
\end{proof}

\begin{corollary}
\label{bootstrapalpha1}
 Let $R \in X^{k+\log}$. Then $R \in X^{k+1+\log}$.
\end{corollary}
\begin{proof}
%

Let $F_m$ be the space that consists of functions that have modes $> m$. It is immediate that $LI$ maps $F_m$ into $F_m$, but it is not obvious that $S$ does the same. Instead, to ensure this, we apply a projection operator $P_{>m}$ onto the modes greater than $m$ to equation \eqref{ecuaciontildeISG} to obtain

\begin{align*}
 (P_{>m}LI + P_{>m}S)(\pa^{k-1}R^{high}) = P_{>m}\tilde{g}(R).
\end{align*}

Now, both $P_{>m}LI$ and $P_{>m}S$ map $F_m$ into $F_m$ and $ (P_{>m}LI + P_{>m}S)$ is invertible. Therefore

\begin{align*}
 \pa^{k-1} R^{high} = \underbrace{(P_{>m}LI+P_{>m}S)^{-1}\underbrace{P_{>m}\tilde{g}(R)}_{\in H^{1}}}_{\in X^{2+\log}} \in X^{2+\log},
\end{align*}

which implies $R^{high} \in X^{k+1+\log}$. Finally

\begin{align*}
 \|R\|^{2}_{X^{k+1+\log}} = \underbrace{\|R^{high}\|^{2}_{X^{k+1+\log}}}_{\text{already shown } < \infty} + \underbrace{\|R^{low}\|^{2}_{X^{k+1+\log}}}_{\text{finite sum of finite coefficients since }R \in X^{k+\log}} < \infty,
\end{align*}

which implies that $R \in X^{k+1+\log}$.

\end{proof}

We conclude this section with a proposition concerning the convexity of the patches.

\begin{prop}
\label{convexitypatches}
 Let $r$ be small enough. Then the set of solutions constructed in the previous section parametrizes convex patches.
\end{prop}

\begin{proof}
 We compute the signed curvature at a point $x$:

\begin{align*}
 \displaystyle \kappa (x) & =  \frac{(R(x))^2 + 2(R'(x))^2 - R(x) R''(x)}{((R(x))^2 + (R'(x))^2)^{3/2}} > \frac{R(x)(R(x)-R''(x))}{((R(x))^2 + (R'(x))^2)^{3/2}} \\
&\displaystyle > \frac{\min_{x}R(x)(\min_{x}R(x) - \max_{x} R''(x))}{((R(x))^2 + (R'(x))^2)^{3/2}} > 0
\end{align*}

if $r$ is small enough. This shows the convexity.

\end{proof}

Combining Corollaries \ref{bootstrapalphamenor1}, \ref{bootstrapalphamayor1} and \ref{bootstrapalpha1}, and Proposition \ref{convexitypatches}, we derive Theorem \ref{teoremaregularidadconvexidad}.

\section*{Acknowledgements}

AC, DC and JGS were partially supported by the grant MTM2011-26696 (Spain) and ICMAT Severo Ochoa project SEV-2011-0087. AC was partially supported by the ERC grant 307179-GFTIPFD. We are very grateful to Charlie Fefferman for his suggestions about the regularity proof for $\al = 1$.

 \bibliographystyle{abbrv}
 \bibliography{references}

%

\begin{tabular}{l}
\textbf{Angel Castro} \\
{\small Instituto de Ciencias Matem\'aticas}\\
{\small Universidad Aut\'onoma de Madrid}\\
{\small C/ Nicolas Cabrera, 13-15, 28049 Madrid, Spain}\\
{\small Email: angel\underline{  }castro@icmat.es}\\
    \\
  \textbf{Diego C\'ordoba} \\
  {\small Instituto de Ciencias Matem\'aticas} \\
 {\small Consejo Superior de Investigaciones Cient\'ificas} \\
 {\small C/ Nicolas Cabrera, 13-15, 28049 Madrid, Spain} \\
  {\small Email: dcg@icmat.es} \\
\\
 {\small Department of Mathematics} \\
 {\small Princeton University} \\
 {\small 804 Fine Hall, Washington Rd,} \\
  {\small Princeton, NJ 08544, USA} \\
 {\small Email: dcg@math.princeton.edu} \\
\\
\textbf{Javier G\'omez-Serrano} \\
{\small Department of Mathematics} \\
{\small Princeton University}\\
{\small 610 Fine Hall, Washington Rd,}\\
{\small Princeton, NJ 08544, USA}\\
 {\small Email: jg27@math.princeton.edu}
  \\

\end{tabular}

\end{document}